\documentclass[10pt]{amsart}
\usepackage{amsmath}
\usepackage{amssymb,amscd}
\usepackage{graphicx}
\usepackage{mathdots}
\usepackage{mathrsfs}       % for the script X used by \XXS
\usepackage{graphicx}
\usepackage{bbm} % for mathbbm
\usepackage{microtype} % to adjust spacing between characters to reduce the number of bad line breaks.

\usepackage{colonequals} %\colonequal
\usepackage{mathtools}% \vcentcolon=

\usepackage{tikz,tkz-euclide}
\usepackage{tikz-cd}
\usetikzlibrary{arrows}
\usepackage{tikz-3dplot}

\usepackage{color}\definecolor{darkblue}{rgb}{0,0.1,.5}
\usepackage[colorlinks=true,linkcolor=darkblue, urlcolor=darkblue, citecolor=darkblue]{hyperref}
\usepackage{cleveref}
\usepackage{comment}

\theoremstyle{plain}
\newtheorem{theorem}{Theorem}[section]
\newtheorem{lemma}[theorem]{Lemma}
\newtheorem{proposition}[theorem]{Proposition}
\newtheorem{corollary}[theorem]{Corollary}
\newtheorem{problem}[theorem]{Problem}

\theoremstyle{definition}
\newtheorem{definition}[theorem]{Definition}
\newtheorem{example}[theorem]{Example}

\newtheorem{construction}[theorem]{Construction}

\theoremstyle{remark}
\newtheorem*{remark}{Remark}

\numberwithin{equation}{section}

\def \begineq{\begin{equation}}
\def \endeq{\end{equation}}

\def \bb{\mathbb}

\def \NN{{\bb{N}}}

\def \ZZ{{\bb{Z}}}

\def\Bier{\mathrm{Bier}}

\def\pol{\mathrm{pol}}

\def\zk{\mathcal Z_K}

\def\Q{\mathbb Q}
\def\R{\mathbb R}

\def\N{\mathbb N}

\DeclareMathAlphabet{\mathbbmsl}{U}{bbm}{m}{sl}

\DeclareMathAlphabet{\mathpzc}{OT1}{pzc}{m}{it}

\newcommand{\gen}[3]{{\mathpzc{#1}}_{#2}^{#3}}

\newcommand{\noopsort}[1]{}
\makeatletter
\@namedef{subjclassname@2020}{%
  \textup{2020} Mathematics Subject Classification}
\makeatother

\title[Neighborly Murai spheres and the Simplicial Steinitz Problem]{Neighborly Murai spheres and the Simplicial Steinitz Problem}

\author[Limonchenko]{Ivan Limonchenko}
\address{Mathematical Institute of the Serbian Academy of Sciences
and Arts (SASA), Belgrade, Serbia}
\email{ivan.limoncenko@turing.mi.sanu.ac.rs}

\author[Vavpeti\v{c}]{Ale\v{s} Vavpeti\v{c}}
\address{University of Ljubljana, Faculty of Mathematics and Physics, Slovenia\newline
\indent Institute of Mathematics, Physics, and Mechanics, Ljubljana, Slovenia}
\email{ales.vavpetic@fmf.uni-lj.si}


\begin{document}

\begin{abstract}
We provide a classification of neighborly Murai spheres, which implies that all of them are polytopal. Furthermore, we show that each neighborly $d$-sphere with no more than $d+4$ vertices is combinatorially equivalent to a Murai sphere for any $d\geq 1$. 
\end{abstract}

\maketitle

\section{Introduction}

For any abstract simplicial complex $K$ on $[m]=\{1,2,\ldots,m\}$ different from the whole simplex $\Delta_{[m]}$ with $m$ vertices, its Bier sphere $\Bier(K)$ was defined in~\cite{Bier} as a deleted join of $K$ and its (combinatorial) Alexander dual $K^\vee$. This beautiful and powerful construction was also shown to provide a $(m-2)$-dimensional PL-sphere with a number of vertices that varies between $m$ and $2m$, see~\cite{Matousek, Longueville}. Since the 1990s, it has been studied intensively and found numerous applications in such areas of research as topological combinatorics~\cite{Matousek}, geometrical combinatorics~\cite{BPSZ05}, polytope theory~\cite{Zivaljevic19,Zivaljevic21}, game theory~\cite{Zivaljevic23}, combinatorial commutative algebra~\cite{HK,CYY,LZ}, toric geometry and toric topology~\cite{LS,LV1,LZ,LTZ1,LTZ2,CYY}.

This paper is a continuation of the research project initiated by~\cite{LV2} and  devoted to studying combinatorial properties of a class of simplicial spheres introduced in~\cite{Mu} under the name 'generalized Bier spheres'; we call them Murai spheres. Given a $c$-multicomplex $M$ with $c\in\N^m, m\geq 1$, different from the set of all monomials in the polynomial algebra $S=\Bbbk[x_1,x_2,\ldots,x_m]$, its generalized Bier sphere $\Bier_c(M)$ was defined in~\cite{Mu} as the boundary of a certain simplicial ball on $[|c|+m]$, $|c|=c_1+c_2+\ldots+c_m$. In the special case $c=(1,1,\ldots,1)$, the complex $\Bier_c(M)$ is isomorphic, or combinatorially equivalent, to $\Bier(K_M)$, where $K_M$ is the simplicial complex corresponding to $M$. Moreover, $\Bier_c(M)$ was proven to be a $(|c|-2)$-dimensional simplicial sphere, whose number of vertices ranges between $|c|$ and $|c|+m$, see~\cite{Mu}.

It was also shown in~\cite{Mu} that Murai spheres share certain important combinatorial properties with classical Bier spheres: all of them are shellable and edge decomposable simplicial spheres~\cite{BPSZ05, Mu}. This similarity was extended further in~\cite{LZ}, where it was proved that the class of flag Bier spheres coincides with the class of flag Murai spheres; all of them are flag nestohedra, and therefore, are polytopal. The classical Simplicial Steinitz problem---whether a given simplicial sphere admits a polytopal realization---is open not only for Bier spheres, but even for Murai spheres. There exist infinitely many non-polytopal Bier spheres~\cite{Matousek}, which immediately implies the existence of infinitely many non-polytopal Murai spheres as well; however, no explicit example of a non-polytopal Murai sphere has been constructed so far. 

In this paper, we continue the studies of the combinatorial structure of Murai spheres with the aim of identifying the widest possible class of polytopal Murai spheres, see~\cite{LV2}, where we classified all the (combinatorial types of) Murai spheres in dimensions one and two. These spheres form a subclass of the class of all neighborly spheres. Recall that a $d$-sphere is called neighborly if each set of its vertices having cardinality no greater than $\lceil\frac{d}{2}\rceil$ forms a face. It turns out that in dimensions $d\geq 3$ neighborly spheres might be non-polytopal, as is the famous Br\"uckner sphere of dimension 3 with 8 vertices, since its graph is complete due to~\cite{GrSr}. The case of higher dimensions is addressed in~\cite[Theorem 4.3]{NoZh}. Furthermore, it is also known, see~\cite{Grun}, that if $n$ is even, then all
neighborly $n$-polytopes are simplicial; if $n$ is odd, then for each $m\geq n+2$ there
are neighborly $n$-polytopes with $m$ vertices that are not simplicial. Here, we are concerned with the class of all neighborly Murai spheres and their combinatorial type classification.

The structure of this paper is as follows. Section 2 is devoted to introducing the key definitions, constructions and results on Murai spheres that are needed in the subsequent parts of the manuscript. In particular, we recall the above mentioned classification of low-dimensional Murai spheres obtained in~\cite{LV2} and prove the necessary and sufficient conditions describing ghost vertices of a Murai sphere.

In Section 3, we deal with an algebraic classification of the most general class of Murai spheres under consideration in this paper; namely, that of the $k$-neighborly $d$-spheres. That is, any $k$ vertices of such a sphere form its face. Unless the sphere is the boundary of a simplex, we must have $1\leq k\leq \lceil\frac{d}{2}\rceil$. For each $k$ under these constraints, we obtain a complete description of the resulting integer vectors $c$ and the corresponding proper $c$-multicomplexes $M$. Two particular cases are considered in detail: $k=1$ is the case where the graph of the sphere is complete and $k=\lceil\frac{d}{2}\rceil$ is the case where the sphere is neighborly. The latter case is the main one in our studies and our algebraic classification of neighborly Murai spheres is given below.

\begin{theorem}
A Murai sphere $\Bier_c(M)$ is neighborly if and only if $M$ or its dual $M^\vee$ is from the following list:
\begin{itemize}
\item[(a)] $c=(c_1,1)$ and $M=\langle x_1^{c_1}\rangle$;
\item[(b)] $c=(k,k,1)$ and $M=\langle x_1^kx_2^k\rangle$;
\item[(c)] $c=(k+1,k,1)$ and $M=\langle x_1^{k+1}x_2^k\rangle$;
\item[(d)] $c=(2,2,2)$ and $M_3^2\subseteq M\subseteq M_3^3$;
\item[(e)] $c=(k+2,k)$, $(k+1,k+1)$,  $(k+1,k,1)$, $(k,k,2)$,  $(k,k,1,1)$ and $M_2^k\subseteq M\subseteq M_2^{k+1};$ 
\item[(f)] $c=(k+1,k)$, $(k,k,1)$ and $M=M_2^k$;
\item[(g)] $c=(c_1,c_2,\ldots,c_m)$, where $c_1\ge k$, $|c|=2k+1$ or $|c|=2k+2$ and $M=\langle x_1^a\rangle$, where $k\le a\le |c|-k-1$.
\end{itemize}
\end{theorem}

Here, we use the notation $M_r^p:=\langle x_1^{a_1}\ldots x_r^{a_r}\mid a_1+\cdots+a_r=p\rangle$ which means the $c$-multicomplex generated by the set of all monomials $x^a$ such that $|a|=p$. The rest of the paper is devoted to describing the combinatorial types of the Murai spheres listed in the above theorem and proving that all of them are polytopal. 

It turns out that we can achieve this goal by means of the simplicial multiwedge construction and using the notion of a cyclic polytope. These are recalled in Section 4, where we deal with the class of polytopal $(n-1)$-dimensional spheres with $m$ vertices, which can be obtained from the boundary $\Delta(2p+3,2p)$ of a cyclic polytope of dimension $2p$ with $(2p+3)$ vertices by applying the simplicial multiwedge construction with the integer vector $J\in\N^{2p+3}$. We denote such a sphere by $\Delta(p,J)$. In particular, for each such sphere one has $m-n=3$. In fact, each simplicial $(n-1)$-sphere with $(n+3)$ vertices is polytopal, see~\cite{Mani}, and moreover, is isomorphic to either a join of three simplices or $\Delta(p,J)$ for some positive integer $p$ and integer vector $J\in\N^{2p+3}$, see~\cite{Er}. For the sake of completeness, we provide an elementary proof of the fact that $\Delta(p,J)$ is a neighborly sphere if and only if $J$ is a $\{1,2\}$-vector satisfying the Evenness condition: if the components of $J$ are viewed in cyclic order, then between any two consecutive '2's there is an even number of '1's. In particular, a vector $J=(1,\ldots,1)$ satisfies the Evenness condition. This fact can also be proved using Gale diagrams; it yields the classification of all neighborly $(d-1)$-spheres with $d+3$ vertices in the following form.

\begin{theorem}
Let $d\geq 2$. Then $K$ is a neighborly $(d-1)$-sphere with $d+3$ vertices if and only if one of the following two conditions holds:
\begin{itemize}
\item $K\cong\partial\Delta^1\ast\partial\Delta^1\ast\partial\Delta^1$ (the boundary of octahedron);
\item there exist $p\in\N$ and $J\in\N^{2p+3}$ with $|J|=d+3$ satisfying the Evenness condition such that 
$K$ is isomorphic to $\Delta(p,J)$.
\end{itemize}
\end{theorem}

In Section 5, we apply the methods and results worked out in the previous section to identify the class of Murai spheres described in Theorem 1.1 (e) and (f). These are precisely the neighborly Murai $(d-1)$-spheres with $(d+3)$ vertices and $d\geq 4$ being odd (case (e)) or even (case (f)), respectively. On the other hand, it is known that when $d$ is even, a neighborly $(d-1)$-sphere with $d+3$ vertices is isomorphic to $\Delta(d+3,d)$, the boundary of the corresponding cyclic polytope. It turns out that the latter family of neighborly spheres coincides with the one determined in Theorem 1.1 (f), which generalizes~\cite[Theorem 5.3]{LV2}: each boundary of a cyclic polytope of the type $\Delta(d+3,d), d\geq 2$ is isomorphic to a Murai sphere. As an application of our technique and Example~\ref{MuraiWithMequalsOneExample}, we obtain the next general result.

\begin{theorem}
Let $d\geq 2$. Then any neighborly simplicial $(d-1)$-sphere with no more than $d+3$ vertices is isomorphic to a Murai sphere.     
\end{theorem}

Finally, in Section 6, we combine all the previously obtained results and use the classification of neighborly simplicial 5-polytopes with 9 vertices obtained in~\cite{Fin} in order to complete the geometrical classification of neighborly Murai spheres. Since in dimensions 1 and 2 each sphere is neighborly and low-dimensional Murai spheres were already classified in~\cite{LV2}, it remains to consider neighborly Murai spheres in dimensions greater than or equal to 3. This leads us to the following statement, which resolves~\cite[Problem 5.4]{LV2}.

\begin{theorem}
Let $n\geq 4$. An $(n-1)$-dimensional Murai sphere is neighborly if and only if it is isomorphic to one of the following polytopal spheres:
\begin{itemize}
\item $\Delta(n+1,n), n\geq 4$, the boundary of a simplex;
\item $\Delta(n+2,n),n\geq 4$, a join of two boundaries of simplices whose dimensions differ by not more than one;
\item the boundary of a simplicial neighborly 5-polytope with 9 vertices described in Example~\ref{NeighborlyMuraiTwoTwoTwoExample};
\item $\Delta(n+3,n)(J),n\geq 4$, where $n$ is even and $J$ satisfies the \emph{Evenness condition}.
\end{itemize}
In particular, each neighborly Murai sphere is polytopal.
\end{theorem}

On the other hand, it follows from~\cite[Lemma 2.10]{LV2} that a Bier sphere of dimension greater than two is neighborly if and only if it is the boundary of a simplex. As an immediate consequence of the last theorem, we get a complete description of all the boundaries of cyclic polytopes $\Delta(a,b)$ that are isomorphic to Murai spheres, see Corollary~\ref{ClassificationCyclicMuraiCoro}. Finally, we discuss some open problems related to our main results.

%%%%%%%%%%%%%%%%%%%%%%%%%%%%%%%%%%%%%%%%%%%%%
\section{Basic definitions, examples and constructions}
%%%%%%%%%%%%%%%%%%%%%%%%%%%%%%%%%%%%%%%%%%%%%

In this section, we are going to define the class of generalized Bier spheres, or Murai spheres, introduced in~\cite{Mu} and discuss some examples and constructions necessary in the subsequent sections. Throughout the section we follow the notation used in~\cite{LV2}. 

Let $m\geq 1$, $S:=\Bbbk[x_1,x_2,\ldots,x_m]$, where $\Bbbk$ is a commutative ring with unit, and $c=(c_1,\ldots,c_m)\in\N^m$. 

\begin{definition}
A monomial $x^a:=x_1^{a_1}\cdots x_{m}^{c_m}\in S$ is called a \emph{c-monomial} if $a_i\leq c_i$ for each $1\leq i\leq m$. A \emph{$c$-multicomplex} $M$ is a subset of $c$-monomials in $\Bbbk[x_1,\ldots,x_m]$ such that 
$$
x^a\mid x^b\text{ and }x^b\in M\Rightarrow x^a\in M.
$$
Its {\emph{Alexander dual with respect to $c$}} is a $c$-multicomplex $M^\vee$ defined by
$$
M^\vee := \{(x^{a})^c\,|\,x^{a}\text{ is a }c\text{-monomial such that }x^{a}\notin M\},
$$
where
$$
(x^{a})^c:=x_1^{c_1-a_1}\cdots x_m^{c_m-a_m}.
$$
A $c$-multicomplex $M$ is called \emph{full} if it contains all possible $c$-monomials in $S$; $M$ is called \emph{proper} otherwise.
\end{definition}

Elements of a multicomplex $M$ that are \emph{maximal} with respect to the divisibility partial order relation on $M$ are called \emph{generators} of $M$; their set is denoted by $\max(M)$. Similarly, we define \emph{minimal non-elements} of $M$; their set is denoted by $\min(M)$. The same notation is used in the case of a simplicial complex, since it can be considered as a particular case of a multicomplex, see the above remark. Finally, we use angle brackets to indicate the fact that a set of $c$-monomials generates a $c$-multicomplex; in particular, we write:
$$
M=\langle m\,|\,m\in\max(M)\rangle.
$$

A crucial particular case of the above construction is the notion of (abstract) simplicial complex. Namely, if $c=(1,1,\ldots,1)$, then  $c$-multicomplexes $M$ are in one-to-one correspondence with simplical complexes $K_M$ on $[m]:=\{1,2,\ldots,m\}$. Likewise, any simplicial complex $K$ is determined by the set $\max(K)$ of its maximal faces (\emph{facets}) with respect to inclusion, as well as by the set of its \emph{minimal non-faces} $\min(K)$ with respect to inclusion. Furthermore, (combinatorial) \emph{Alexander dual} to $K$ is a simplicial complex $K_{M^\vee}$.

Here we are interested in a certain class of triangulated spheres called Bier spheres and its generalizations. Recall that a complex $K$ is a $d$-dimensional \emph{triangulated sphere} if its geometric realization $|K|$ is homeomorphic to the unit sphere in $\R^{d+1}$. In~\cite{Bier}, it was shown that a deleted join of any simplicial complex $K\neq\Delta_{[m]}$ on $[m]$ with its combinatorial Alexander dual complex $K^\vee$ yields a triangulated sphere $\Bier(K)$ called its \emph{Bier sphere}. 

We are now ready to give a definition of the main object of study in this paper, which generalizes the notion of a Bier sphere in the direction of multicomplexes.  

\begin{definition}[\cite{Mu}]
Let $M$ be a proper $c$-multicomplex. Its \emph{Murai sphere} $\Bier_c(M):=\partial B_c(M)$ is defined to be the boundary of its \emph{Murai ball} $B_c(M)$, which is a simplicial complex on $\tilde{X}:=\tilde{X}_1\cup\ldots\cup\tilde{X}_m$, $\tilde{X}_i:=\{\gen{x}{i}{0},\ldots,\gen{x}{i}{c_i}\}$ for $1\leq i\leq m$, and
$$
\max(B_c(M)):=\{F_c(x^a)\mid x^a\in M\}\text{ for }F_c(x^a):=\tilde{X}\setminus\{\gen{x}{1}{a_1},\ldots,\gen{x}{m}{a_m}\}.
$$
\end{definition}

The term introduced above is justified by the next results, see~\cite{Mu}. For any proper $c$-multicomplex $M$:

\begin{itemize}
\item for $c=(1,\ldots,1)$, one gets a Bier sphere: $\Bier_c(M)\cong\Bier(K_M)$;
\item $\Bier_c(M)\cong\Bier_c(M^\vee)$ is a triangulated $(|c|-2)$-sphere;
\item $\Bier_c(M)$ is shellable and edge decomposable.  
\end{itemize}

Using the above definition, one can also describe a Murai sphere in terms of its facets as follows:
$$
\max(\Bier_c(M))=\{F_c(x^a)\setminus\{\gen{x}{i}{j}\}\mid x^a\in M,x^a\diamond x_i^j\not\in M, a_i<j\le c_i\},
$$
where the $\diamond$ operation is defined as follows:
$$
x^a\diamond x_i^j:=x_1^{a_1}\cdots x_{i-1}^{a_{i-1}}x_i^j x_{i+1}^{a_{i+1}}\cdots x_m^{a_m}.
$$

\begin{example}\label{MuraiWithMequalsOneExample}
If $m=1$, then $c\in\N$ and $M=\langle x^a\rangle$ for some $0\leq a\leq c-1$. Then 
$$
\Bier_c(M)=\partial\Delta^{a}\ast\partial\Delta^{c-a-1}.
$$

Indeed, each facet of $\Bier_c(M)$ looks like 
$$
F_c(x^i)\setminus\{\gen{x}{}{j}\}=\{\gen{x}{}{0},\ldots,\gen{x}{}{i-1},\gen{x}{}{i+1},\ldots,\gen{x}{}{j-1},\gen{x}{}{j+1},\ldots,\gen{x}{}{c}\}
$$ 
with $0\leq i\leq a$ and $a+1\leq j\leq c$.
\end{example}

\begin{remark}
The previous example shows that any join of two boundaries of simplices is a Murai sphere. On the other hand, due to the result of~\cite{LV2}, a join of two boundaries of simplices is a Bier sphere if and only if at least one of those simplices has dimension $\leq 1$.    
\end{remark}

Further examples are formed by the next classification of low-dimensional Murai spheres.

\begin{theorem}[\cite{LV2}]\label{LowDimClassificationThm}
The following two statements hold.
\begin{itemize}
\item Each 1-dimensional Murai sphere is isomorphic to a Bier sphere;
\item Each 2-dimensional Murai sphere is isomorphic to either a Bier sphere, or to $\Bier_{(2,1,1)}(\langle x^2,y,z\rangle)$, or to $\Bier_{(2,1,1)}(\langle x^2,xy,z\rangle)$. 
\end{itemize}

The two exceptional Murai polytopes are shown in Figure~\ref{LowDimMuraiFig}.
\end{theorem}

\begin{figure}\label{LowDimMuraiFig}
\includegraphics[scale=0.8]{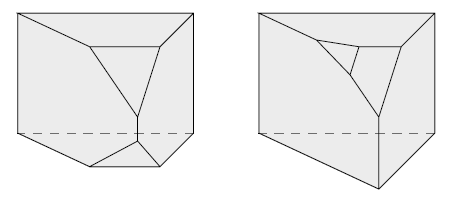}
\caption{Murai polytopes in dimension $3$ that are not isomorphic to Bier polytopes.}
\end{figure}

In the last part of this section, we are going to describe the set $\min(\Bier_c(M))$ of minimal non-faces of a Murai sphere $\Bier_c(M)$. In general, minimal non-faces of $\Bier_c(M)$ correspond to the elements in the minimal set of generators for the Stanley--Reisner ideal of $\Bier_c(M)$. To describe the latter, let us first recall the ideal-theoretic Alexander duality. 

\begin{definition}
Let $I\subset S$ be a $c$-\emph{ideal}, that is, $I$ is generated by $c$-monomials. Its {\emph{Alexander dual with respect to $c$}} is a $c$-ideal $I^\vee\subset S$ defined by
$$
I^\vee = \{(x^{a})^c\,|\,x^{a}\text{ is a }c\text{-monomial in }S\text{ such that }x^{a}\notin I\}.
$$
\end{definition}

Denote by $I_c(M)\subset S$ the $c$-ideal generated by all $c$-monomials not in $M$. Note that for $c=(1,\ldots,1)$, this is the Stanley-Reisner ideal of the simplicial complex $K_M$ corresponding to $M$. Moreover, it is easy to see that
$$
(I_c(M))^\vee = I_c(M^\vee).
$$

Given a monomial ideal $I\subset S$, denote by $G(I)$ the unique minimal set of monomial generators of $I$.
The next observation was proved in~\cite[Lemma 4.3]{LZ}.

\begin{proposition}[\cite{LZ}]\label{MinimalNonMonmomialsProp}
For any $c$-proper multicomplex $M$, one has:
$$
G(I_c(M))=\{(x^a)^c\,|\,x^a\in\max(M^\vee)\}=\min(M)
$$
and 
$$
G(I_c(M)^\vee)=\{(x^a)^c\,|\,x^a\in\max(M)\}=\min(M^\vee).
$$
\end{proposition}

Let us describe the set $\min(\Bier_c(M))$ of minimal non-faces of a Murai sphere $\Bier_c(M)$ in terms of its Stanley-Reisner ideal.
First, we define two polarizations for a $\bar{c}$-ideal in $S$, where $\bar{c}:=(c_1+1,\ldots,c_m+1)\in\N^m$. They will both belong to the polynomial algebra on the set of variables 
$$
X := X_1\cup X_2\cup\ldots\cup X_m,\text{ where }X_i := \{x_{i,0},\ldots,x_{i,c_i}\}\text{ for }1\leq i\leq m.
$$

\begin{definition}
The \emph{polarization} of a $\bar{c}$-ideal $I\subset\Bbbk[m]$ is the monomial ideal $\pol(I)$ in $\Bbbk[X]$ 
such that
\[
\pol(I):=(\pol(x^a):=\prod\limits_{a_i\neq 0}x_{i,0}\ldots x_{i,a_{i}-1}\,|\,x^a\in G(I))\subset\Bbbk[X].
\]

The $\ast$-\emph{polarization} of a $\bar{c}$-ideal $I\subset\Bbbk[m]$ is the monomial ideal $\pol^\ast(I)$ in $\Bbbk[X]$ such that
\[
\pol^\ast(I):=(\pol^\ast(x^a):=\prod\limits_{a_i\neq 0}x_{i,c_i}\ldots x_{i,c_i-a_{i}+1}\,|\,x^a\in G(I))\subset\Bbbk[X].
\]
\end{definition}

The next statement proved in~\cite[Theorem 3.6]{Mu} describes the face ideal of $\Bier_c(M)$ and therefore the set $\min(\Bier_c(M))$ of its minimal non-faces.

\begin{theorem}[\cite{Mu}]\label{FaceIdealMuraiThm}
Let $M$ be a proper $c$-multicomplex. Then the Stanley-Reisner ideal of the Murai sphere $\Bier_c(M)$ is equal to the sum of the three ideals:
\[
I_{SR}(\Bier_c(M))=\pol(I_c(M))+\pol^\ast(I_c(M^\vee))+\pol(x_{1,c_{1}+1},\ldots,x_{m,c_{m}+1}).
\]
\end{theorem}

Here is a corollary of the previous theorem, which describes a particularly important class of minimal non-faces of a Murai sphere, its ghost vertices. Namely, we are going to obtain necessary and sufficient conditions for a singleton to be in $\min(\Bier_c(M))$.

\begin{proposition}\label{CharacterisationGhostVerticesProp}
Let $c=(c_1,\ldots,c_m)\in\N^m, m\geq 1$ and $M$ be a proper $c$-multicomplex. Suppose $v$ is a ghost vertex of $\Bier_c(M)$. Then one of the next cases takes place.
\begin{itemize}
\item[(a)] $v=\gen{x}{i}{0}$ for some $1\leq i\leq m$. This holds if and only if $x_i\notin M$;
\item[(b)] $v=\gen{x}{i}{c_i}$ for some $1\leq i\leq m$. This holds if and only if $(x_i)^c\in M$.
\end{itemize}
Furthermore, the following description of ghost vertices of a Murai sphere holds.
\begin{itemize}
\item If $c_i>1$ and $\gen{x}{i}{0}$ is a ghost vertex, then any ghost vertex of $\Bier_c(M)$ has the form $\gen{x}{j}{0}$ for some $1\leq j\leq m$; 
\item If $c_i>1$ and $\gen{x}{i}{c_i}$ is a ghost vertex, then any ghost vertex of $\Bier_c(M)$ has the form $\gen{x}{j}{c_j}$ for some $1\leq j\leq m$. 
\end{itemize}
In particular, $\gen{x}{i}{0},\gen{x}{i}{c_i}\in\min(\Bier_c(M))$ if and only if $M=\langle (x_i)^c\rangle$ and $c_i=1$.
\end{proposition}
\begin{proof}
Note that $v=\gen{x}{i}{j}\in\tilde{X}$ is a ghost vertex of $\Bier_c(M)$ if and only if $x_i^j\in G(I_{SR}(\Bier_c(M)))$. Since $c_i\geq 1$ for all $1\leq i\leq m$, due to Theorem~\ref{FaceIdealMuraiThm} we have $x_{i,j}\in\pol(I_c(M))+\pol^{*}(I_c(M^\vee))$. It remains to apply the definition of polarization to observe that $x_{i,j}\in\pol(I_c(M))$ is equivalent to $j=0$, which is the case if and only if $x_i\notin M$; $x_{i,j}\in\pol(I_c(M^\vee))$ is equivalent to $j=c_i$, which is the case if and only if $x_i\notin M^\vee$, or equivalently, $(x_i)^c\in M$. This proves statements (a) and (b).

The last two statements follow from the fact that if $c_i>1$, then $(x_i)^c$ divides each generator $x_j$ with $1\leq j\leq m$ of the polynomial algebra $S$.

Finally, $\gen{x}{i}{0},\gen{x}{i}{c_i}\in\min(\Bier_c(M))$ if and only if $x_i\notin M$ and $x_i\notin M^\vee$, which is equivalent to $x_i\notin M$ and $(x_i)^c\in M$. This is the case if and only if $c_i=1$ and $M=\langle (x_i)^c\rangle=M^\vee$.
\end{proof}

Denote by $V\subseteq [m]$ the set of indices of ghost vertices of $\Bier_c(M)$. Then Proposition~\ref{CharacterisationGhostVerticesProp} shows that the ghost vertices form the set $\{\gen{x}{i}{0}\,|\,i\in V\}$, or $\{\gen{x}{i}{c_i}\,|\,i\in V\}$, or $\{\gen{x}{i}{0},\gen{x}{i}{1}\}$ if $M=\langle (x_i)^c\rangle$ and $c_i=1$.

%%%%%%%%%%%%%%%%%%%%%%%%%%%%%%%%%%%%%%%%%%%%%%%%%%%%%%%%%%%%
\section{$k$-neighborly Murai spheres}
%%%%%%%%%%%%%%%%%%%%%%%%%%%%%%%%%%%%%%%%%%%%%%%%%%%%%%%%%%%%

As was shown in Example~\ref{MuraiWithMequalsOneExample}, there exist infinitely many Murai spheres (combinatorially) different from Bier spheres such that any two vertices are linked by an edge. In this section, we are going to consider this property of Murai spheres and its generalization in more detail.

The rest of the paper will be devoted to the class of Murai spheres that possess the following property. 

\begin{definition}
Let $k\geq 2$ be an integer. A \emph{$k$-neighborly} complex is a simplicial complex in which every set of $k$ or fewer geometrical vertices forms a simplex. A $d$-sphere is called \emph{neighborly} if it is $\lceil d/2\rceil$-neighborly.
\end{definition}

\begin{example}
Let $c=(k+1,k)$ and $M=\langle x_1^ix_2^{k-i}\,|\, i=0,\ldots,k\rangle$. Then $\Bier_c(M)$ is a neighborly sphere.

Indeed, since $\dim(\Bier_c(M))=2k-1$, we must prove that each $k$ of its vertices form a face. Suppose $\sigma\subset\tilde{X}=\{\gen{x}{1}{0},\ldots,\gen{x}{1}{k+1},\gen{x}{2}{0},\ldots,\gen{x}{2}{k}\}$ has $k$ elements. Then there exists $0\leq \alpha\leq k$ such that $\{\gen{x}{1}{\alpha},\gen{x}{2}{k-\alpha}\}\cap\sigma=\emptyset$; by definition of $M$ we also have $x_1^\alpha x_2^{k-\alpha}\in M$. The set $\rho=\{\gen{x}{1}{\alpha+1},\ldots,\gen{x}{1}{k+1},\gen{x}{2}{k-\alpha+1},\ldots,\gen{x}{2}{k}\}$
has $k+1$ elements, so there exists $\gen{x}{j}{\beta}\in\rho\setminus\sigma$. By definition of the $\diamond$-operation, the degree of the monomial $x_1^\alpha x_2^{k-\alpha}\diamond x_j^\beta$ is greater than $k$, hence $x_1^\alpha x_2^{k-\alpha}\diamond x_j^\beta\not\in M$. Therefore, $G(x_1^\alpha x_2^{k-\alpha};x_j^\beta)$ is a facet of the Murai sphere which contains $\sigma$.
\end{example}

\begin{lemma}\label{NeighbornessAndCcomponentsLemma}
Let $c=(c_1,\ldots,c_m)$ and let $\Bier_c(M)$ be a $k$-neighborly complex. Then $c_i\ge k$, $x_i^k\in M$, and $x_i^k\in M^\vee$ for all $i\not\in V$. If $\Bier_c(M)$ is $k$-neighborly without ghost vertices and $m=2$, then at least one of the $c_i$ must be greater than $k$.
\end{lemma}
\begin{proof}
Let $i\not\in V$ and $c_i<k$. Then the set $\{\gen{x}{i}{0},\ldots,\gen{x}{i}{c_i}\}$ has no more than $k$ elements, but it is not contained in any facet $G(x_1^{\alpha_1}\cdots x_m^{\alpha_m};x_j^\beta)$. If $\Bier_c(M)$ is $k$-neighborly, then $x_i^k\in M$, since otherwise $\{\gen{x}{i}{0},\ldots,\gen{x}{i}{k-1}\}$ is a non-face in $\Bier_c(M)$ by Theorem~\ref{FaceIdealMuraiThm}, a contradiction. Since $\Bier_c(M^\vee)\cong\Bier_c(M)$, we also have $x_i^k\in M^\vee$.

For $m=2$, we obtain $x_1^{c_1}x_2^{c_2-k}=(x_2^k)^c\not\in M$, since from above we have $x_2^k\in M^\vee$. Hence, $c\ne(k,k)$, so at least one of the $c_i$ must be greater than $k$.    
\end{proof}

\begin{theorem}
Let $c=(c_1,\ldots,c_m)$, let $\Bier_c(M)$ be without ghost vertices, and let $k\ge 2$.
The Murai sphere $\Bier_c(M)$ is $k$-neighborly if and only if $c_i\ge k$ for all $i$ and for every $a=(a_1,\ldots,a_m)$ such that $|a|:=a_1+\cdots+a_m\le k$ we have $x^a\in M$ and $(x^a)^c\not\in M$.
\end{theorem}

\begin{proof}
Let $\Bier_c(M)$ be $k$-neighborly. By Lemma~\ref{NeighbornessAndCcomponentsLemma}, we have $c_i\ge k$ for all $i$. Suppose $a=(a_1,\ldots,a_m)$ is such that $|a|\le k$. Then the set $\sigma=\{\gen{x}{i}{j}\mid 1\le i\le m,0\le j<a_i\}$ has $k$ elements, so it is a simplex in the Murai sphere. Hnence there exists a facet $G(x^b;x_i^\beta)$ containing $\sigma$. By the definition of the Murai sphere, we have $x^b\in M$, and by the definition of $\sigma$, we have $b_j\ge a_j$ for all $j$, hence $x^a\in M$. Since $\Bier_c(M^\vee)\cong\Bier_c(M)$ is $k$-neighborly, the same argument shows that $x^a\in M^\vee$. By definition of the Alexander-dual multicomplex, this means that $(x^a)^c\notin M$.

Suppose now that we have $c_i\ge k$ for all $i$ and for every $a=(a_1,\ldots,a_m)$ such that $|a|\le k$ we have $x^a\in M$ and $(x^a)^c\not\in M$. Let $\sigma$ be an arbitrary subset of vertices of the Murai sphere with $k$ elements. Let $k_j$ be the number of vertices in $\sigma$ of the type $\gen{x}{j}{*}$. Then, for every $i$ there exist $a_i,b_i\in\{0,\ldots, k_i\}$ such that $\gen{x}{i}{a_i},\gen{x}{i}{c_i-b_i}\not\in\sigma$; it may happen that $a_i=c_i-b_i$. Let us define $a=(a_1,\ldots,a_m)$ and $b=(b_1,\ldots,b_m)$. Since $|a|\le k_1+\cdots+k_m=k$ we have $x^a\in M$. Let $a^j=(c_1-b_1,\ldots,c_{j}-b_{j},a_{j+1},\ldots,a_m)$. Then for $j=0$ we have $x^{a^0}=x^a\in M$ but for $j=m$ we have $x^{a^{m}}=(x^b)^c\not\in M$, since $|b|\le k_1+\cdots+k_m=k$. Hence there exists $j\in\{1,\ldots, m\}$ such that $x^{a^{j-1}}\in M$ but $x^{a^{j}}\not\in M$. Then $G(x^{a^{j-1}};x_j^{c_j-b_j})$ is a facet in the Murai sphere containing the set $\sigma$, hence the Murai sphere $\Bier_c(M)$ is $k$-neighborly.
\end{proof}

\begin{theorem}\label{NeighborlySpheresThm}
Let $c=(c_1,\ldots,c_m)$ and $k\ge 2$.
\begin{itemize}
\item[(a)] Suppose there exists $i$ such that $\gen{x}{i}{0}$ and $\gen{x}{i}{c_i}$ are ghost vertices of the Murai sphere $\Bier_c(M)$.
Then $\Bier_c(M)$ is $k$-neighborly if and only if $c_j\ge k$ for all $j\neq i$.

\item[(b)] Suppose there exists $V\subset [m]$ such that $\{\gen{x}{i}{0}\mid i\in V\}$ is the set of ghost vertices of $\Bier_c(M)$.
Then $\Bier_c(M)$ is $k$-neighborly if and only if $c_i\ge k$ for all $i\not\in V$ and for every $a=(a_1,\ldots,a_m)$ such that $|a|\le k$ and $a_i=0$ for all $i\in V$ we have $x^a\in M$, and for every $a=(a_1,\ldots,a_m)$ such that $|a|\le k$ we have $(x^a)^c\not\in M$.

\item[(c)] Suppose there exists $V\subset [m]$ such that $\{\gen{x}{i}{c_i}\mid i\in V\}$ is the set of ghost vertices of $\Bier_c(M)$.
Then $\Bier_c(M)$ is $k$-neighborly if and only if $c_i\ge k$ for all $i\not\in V$ and for every $a=(a_1,\ldots,a_m)$ such that $|a|\le k$ we have $x^a\in M$, and for every $a=(a_1,\ldots,a_m)$ such that $|a|\le k$ and $a_i=0$ for all $i\in V$ we have $(x^a)^c\not\in M$.
\end{itemize}
\end{theorem}

\begin{proof}
(a) By Proposition~\ref{CharacterisationGhostVerticesProp}, $c_i=1$ and $M=\langle (x_i)^c\rangle=M^\vee$. Hence Theorem~\ref{FaceIdealMuraiThm} implies that
$$
I_{SR}(\Bier_c(M))=(x_i^0)+(x_i^1)+\pol(x_1^{c_1+1},\ldots,x_{m}^{c_m+1}).
$$
It follows that 
$$
\min(\Bier_c(M))=\{\{\gen{x}{i}{0}\},\{\gen{x}{i}{1}\},\{\gen{x}{j}{0},\ldots,\gen{x}{j}{c_j}\}\,|\,1\leq j\neq i\leq m\},
$$
which proves statement (a).

(b) By Proposition~\ref{CharacterisationGhostVerticesProp}, $x_i\not\in M$ for all $i\in V$ and $x_j\in M$ for all $j\notin V$.

Let $\Bier_c(M)$ be $k$-neighborly. Then due to Lemma~\ref{NeighbornessAndCcomponentsLemma} we obtain $c_i\geq k$ for all $i\notin V$.

For every $a=(a_1,\ldots,a_m)$ such that $|a|\le k$ and $a_i=0$ for $i\in V$, the set $\sigma=\{\gen{x}{i}{j}\mid 1\le i\le m,0\le j<a_i\}$ is a simplex in the Murai sphere with $k$ vertices, so there exists a facet $G(x^b;x_i^\beta)$ containing $\sigma$. By definition of Murai sphere, $x^b\in M$ and $b_j\ge a_j$ for all $j$, hence $x^a\in M$. 

Let $a=(a_1,\ldots,a_m)$ such that $|a|\le k$. The set $\sigma=\{\gen{x}{i}{c_i-j}\mid 1\le i\le m, 0\le j<a_i\}$ is a simplex in the Murai sphere with $k$ vertices, so there exists a facet $G(x^b;x_i^\beta)$ containing $\sigma$. By  definition of facet of a Murai sphere, $b_j\le c_j-a_j$ for all $j$, $\beta\le c_i-a_i$, and $x^b\diamond x_i^\beta\not\in M$, so $(x^a)^c\not \in M$.

Suppose now that for every $a=(a_1,\ldots,a_m)$ such that $|a|\le k$ and $a_i=0$ for $i\in V$ we have $x^a\in M$ and for every $a=(a_1,\ldots,a_m)$ such that $|a|\le k$ we have $(x^a)^c\not\in M$.
Let $\sigma$ be an arbitrary subset of vertices of the Murai sphere with $k$ elements and $\gen{x}{i}{0}\not\in\sigma$ for $i\in V$. Let $k_j$ be the number of vertices in $\sigma$ of the type $\gen{x}{j}{*}$.
Then, for every $i\notin V$ there exist $a_i,b_i\in\{0,\ldots k_i\}$ such that both $\gen{x}{i}{a_i},\gen{x}{i}{c_i-b_i}\not\in\sigma$ (it is possible that $a_i=c_i-b_i$) and set $a_i=b_i=0$ for $i\in V$. Since $|a|\le k_1+\cdots+k_m=k$ we have $x^a\in M$. Let $a^j=(c_1-b_1,\ldots,c_{j}-b_{j},a_{j+1},\ldots,a_m)$. Then for $a^0=a$ we have $x^{a^0}\in M$ but for $b=(b_1,\ldots,b_m)$ we have $x^{a^{m}}=(x^b)^c\not\in M$, since $|b|\le k_1+\cdots+k_m=k$. Hence there is $1\leq j\leq m$ such that $x^{a^{j-1}}\in M$, but $x^{a^{j}}\not\in M$. Then $G(x^{a^{j-1}};x_j^{c_j-b_j})$ is a facet in the Murai sphere containing the set $\sigma$, hence the Murai sphere $\Bier_c(M)$ is $k$-neighborly.

(c) It follows from (b), since $\Bier_c(M^\vee)\cong \Bier_c(M)$
and, by definition of the Alexander-dual multicomplex, $x_i\not\in M$ if and only if $(x_i)^c\in M$. Hence it suffices to observe that $\gen{x}{i}{0}$ is a ghost vertex in $\Bier_c(M)$ if and only if $\gen{x}{i}{c_i}$ is a ghost vertex in $\Bier_c(M^\vee)$. This finishes the proof of the theorem.
\end{proof}

In the rest of this section, we consider the two opposite cases, where a $d$-dimensional Murai sphere is $k$-neighborly: $k=2$ and $k=\lceil \frac{d}{2} \rceil$, i.e. the neighborness case.
We start with $k=2$.

\begin{example}\label{MuraiWithCompleteGraphEx}
Let $c=(c_1,\ldots,c_m)$, where $c_i\ge 2$ for all $i$. Let $M=\langle x_1^2,\ldots,x_m^2,x_1\cdots x_m\rangle$. If $m\ge 3$ or $m=2$ and $c\ne (2,2)$ then the 1-skeleton of the Murai sphere $\Bier_c(M)$ is the complete graph without ghost vertices, i.e., the complete graph on $(c_1+1)+\cdots+(c_m+1)=|c|+m$ vertices. Indeed, let us show that for every pair of vertices $\gen{x}{i}{\alpha},\gen{x}{j}{\beta}\in V(\Bier_c(M))$, the edge $(\gen{x}{i}{\alpha},\gen{x}{j}{\beta})$ forms a simplex in the Murai sphere. First, consider the case $m\ge 3$. There is $k\in\{1,\ldots,m\}\setminus\{i,j\}$. If $\alpha\ne0$ let $\beta\ne\gamma\in\{1,2\}$. Then $(\gen{x}{i}{\alpha},\gen{x}{1j}{\beta})\in G(x_j^\gamma;x_k^2)$. If $\alpha=0$ let $\beta\ne\gamma\in\{0,1\}$. Then $(\gen{x}{i}{\alpha},\gen{x}{j}{\beta})\in G(x_ix_j^\gamma;x_k^2)$.

Let $m=2$. We may assume that $c_1>2$. If $\alpha\ne 0$ there exist $\gamma\in\{2,\ldots,c_1\}\setminus\{\alpha\}$ and $\delta\in\{1,2\}\setminus\{\beta\}$. Then $(\gen{x}{i}{\alpha},\gen{x}{j}{\beta})\in G(x_1^0x_2^\delta;x_1^\gamma)$. If $\alpha=0$ then $(\gen{x}{i}{0},\gen{x}{j}{0})\in G(x_1^1x_2^1;x_1^2)$ for $\beta=0$, and $(\gen{x}{i}{0},\gen{x}{j}{\beta})\in G(x_1^1x_2^0;x_1^3)$ for $\beta\ne 0$.
\end{example}

The following statement easily follows from the definition of Murai spheres. 

\begin{theorem}
For $d=1,2$, the only complete graph which is the 1-skeleton of some Murai sphere of dimension $d$ has $d+2$ vertices. For $d\ge 3$ the complete graph $K_n$ is the 1-skeleton of some Murai sphere of dimension $d$ if and only if $n\in\{d+2,\ldots,d+3+\lfloor \tfrac d 2\rfloor\}$.
\end{theorem}

\begin{proof}
For $d=1$, the statement is obvious.
For $d=2$, the statement follows from the classification of 2-dimensional Murai spheres, see Theorem~\ref{LowDimClassificationThm}: a 2-dimensional Murai sphere is 2-connected if and only if it is isomorphic to $\partial\Delta^3$.

Suppose $d\ge 3$. Let us prove the ``only if'' part of the statement. Let $M$ be a $c$-proper multicomplex such that the 1-skeleton of the Murai sphere $\Bier_c(M)$ is the complete graph $K_n$. We may assume that $c_1\ge\ldots\ge c_k>1=c_{k+1}=\ldots =c_m$. By dimension reason, we have $n\ge d+2$. Note that the 1-skeleton of $\Bier_c(M)$ has at least $(m-k)$ ghost vertices, since for every $i\in\{k+1,\ldots,m\}$, by definition, we have $\{\gen{x}{i}{0},\gen{x}{i}{1}\}\not\subset F_c(x^a)$ for all possible $a$. Thus, the complete graph $K_n$ has at most 
\[
(c_1+1)+\cdots+(c_m+1)-(m-k)=|c|+k=d+k+2
\] 
geometric vertices. Observe that $0\leq k\leq m$, hence, for fixed $d$ and $m$, we obtain
\[
n\leq d+m+2\text{ and }n\text{ is maximal possible if and only if }k=m.
\] 
The latter implies $|c|\geq 2m$ and therefore, we finally get $n\leq d+\tfrac 1 2|c|+2=d+\tfrac 1 2(d+2)+2=d+3+\tfrac 1 2d$, hence $n\le d+3+\lfloor \tfrac d 2\rfloor$.

Let us prove the ``if'' part of the statement. Consider the $d$-dimensional Murai sphere $\Bier_c(M)$ defined in Example~\ref{MuraiWithCompleteGraphEx} with 
$$
c=(d+4-2m,2,\ldots,2)\text{ and }m\geq 2.
$$
Since $d+4-2m=|c|-2m+2\geq 2$ and $|c|=d+2\geq 5$, we have $k=m\geq 2$ and the 1-skeleton of $\Bier_c(M)$ is the complete graph $K_{n}$ with $n=d+m+2$ geometric vertices.

This construction shows that for any $m$ with $2\leq m\leq \tfrac 1 2|c|=\tfrac 1 2d+1$ there exists a Murai sphere with the 1-skeleton isomorphic to $K_{d+m+2}$. Finally, the Murai spheres $\Bier(\partial\Delta_{[d+2]})$ and $\Bier_{(d+2)}(\langle x^{2}\rangle)$ have the 1-skeleta isomorphic to $K_{d+2}$ and $K_{d+3}$, respectively. This finishes the proof.
\end{proof}

Here is a criterion for a Murai sphere to have a complete graph as its 1-skeleton.

\begin{proposition}
Let $M$ be a $c$-proper multicomplex and $c\in\N^m$. Then the 1-skeleton of its Murai sphere $\Bier_c(M)$ is not a complete graph if and only if one of the following cases holds.
\begin{itemize}
\item[(a)] there exists $1\leq i\leq m$ such that $(x_i^2)^c\in\max(M)\cup\max(M^\vee)$;
\item[(b)] there exist $1\leq i, j\leq m$ such that $(x_ix_j)^c\in\max(M)\cup\max(M^\vee)$;
\item[(c)] there exists $1\leq i\leq m$ such that $c_i=1$ and $(x_i)^c\notin\max(M)\cup\max(M^\vee)$. 
\end{itemize}
\end{proposition}
\begin{proof}
Direct application of Proposition~\ref{MinimalNonMonmomialsProp} and Theorem~\ref{FaceIdealMuraiThm}.
\end{proof}

Finally, we are going to obtain an algebraic classification of all neighborly Murai spheres. In what follows, we use the notation $M_r^p:=\langle x_1^{a_1}\cdots x_r^{a_r}\mid a_1+\cdots+a_r=p\rangle$ which means the $c$-multicomplex generated by the set of all monomials $x^a$ such that $|a|=p$.

\begin{theorem}\label{NeighborlyExplicitClassificationThm}
A Murai sphere $\Bier_c(M)$ is neighborly if and only if $M$ or its dual $M^\vee$ is from the following list:
\begin{itemize}
    \item[(a)] $c=(c_1,1)$ and $M=\langle x_1^{c_1}\rangle$;
    \item[(b)] $c=(k,k,1)$ and $M=\langle x_1^kx_2^k\rangle$;
    \item[(c)] $c=(k+1,k,1)$ and $M=\langle x_1^{k+1}x_2^k\rangle$;
    \item[(d)] $c=(2,2,2)$ and $M_3^2\subseteq M\subseteq M_3^3$;
    \item[(e)] $c=(k+2,k)$, $(k+1,k+1)$,  $(k+1,k,1)$, $(k,k,2)$,  $(k,k,1,1)$ and $M_2^k\subseteq M\subseteq M_2^{k+1};$ 
    \item[(f)] $c=(k+1,k)$, $(k,k,1)$ and $M=M_2^k$;
    \item[(g)] $c=(c_1,c_2,\ldots,c_m)$, where $c_1\ge k$, $|c|=2k+1$ or $|c|=2k+2$ and $M=\langle x_1^a\rangle$, where $k\le a\le |c|-k-1$.
\end{itemize}
\end{theorem}

\begin{remark}
In case (g) by Proposition~\ref{MinimalNonMonmomialsProp}, we obtain $I_c(M^\vee)=((x_1^a)^c)$ and $\Bier_c(M)=\partial\Delta^{a}*\partial\Delta^{|c|-a-1}$, which yields $\partial\Delta^{k+1}\ast\partial\Delta^{k}$ and $\partial\Delta^{k+1}\ast\partial\Delta^{k+1}$, and hence case (g) induces the same Murai spheres as examples (b) and (c).
\end{remark}

\begin{proof}
First, consider the case where both $\gen{x}{m}{0}$ and $\gen{x}{m}{1}$ are ghost vertices of a neighborly Murai sphere $\Bier_c(M)$. Due to Proposition~\ref{CharacterisationGhostVerticesProp}, this implies that $c_m=1$ and $M=\langle (x_m)^c\rangle=M^\vee$.
By \Cref{NeighborlySpheresThm}, we have $c_i\ge k$ for all $i<m$, where, by definition of neighborness, 
$$
k=\left\lceil \tfrac 1 2(\dim\Bier_c(M)+1)\right\rceil=\left\lceil \tfrac 1 2(|c|-2)\right\rceil=\left\lceil \tfrac 1 2|c|\right\rceil-1\geq\left\lceil \tfrac 1 2((m-1)k+1)\right\rceil-1.
$$
If $m\geq 4$, the right side is greater than or equal to $\left\lceil \tfrac 1 2(3k-1)\right\rceil\geq \left\lceil \tfrac 1 2(2k+1)\right\rceil=k+1$, a contradiction. Hence $1\leq m\leq 3$. Let us consider each of the three remaining cases.

\begin{itemize}
\item If $m=1$, then $c=(1)$ and $M=\langle 1\rangle$, hence $\Bier_c(M)=\varnothing$;

\item If $m=2$, then $c=(c_1,1)$ and we must have $M=\langle x_1^{c_1}\rangle$ and $c_1\geq \left\lceil \tfrac 1 2(c_1-1)\right\rceil$, which is true, and we get the case (a);

\item If $m=3$, then $c_1=c_2=:k$ or $c_1=k+1$, $c_2=:k$.
Therefore $c=(k,k,1)$ and $M=\langle x_1^kx_2^k\rangle$, so $\Bier_c(M)=\partial\Delta^k*\partial\Delta^k$, or $c=(k+1,k,1)$ and $M=\langle x_1^{k+1}x_2^{k}\rangle$ and $\Bier_c(M)=\partial\Delta^{k+1}*\partial\Delta^k$. Thus, we get the cases (b) and (c).
\end{itemize}

Suppose now that for every $i$, at most one of the vertices of the form $\gen{x}{i}{*}$ is a ghost vertex. By Alexander duality for multicomplexes and Theorem~\ref{NeighborlySpheresThm}, we can assume that there is a (possibly empty) set $V\subset[m]$ such that $\{\gen{x}{i}{0}\mid i\in V\}$ is the set of ghost vertices of a neighborly Murai sphere $\Bier_c(M)$. Let $d=\dim\Bier_c(M)+1$ and let $k=\left\lceil \tfrac 1 2d\right\rceil$, then by \Cref{NeighborlySpheresThm}, we have $c_i\ge k$ for all $i\notin V$. We consider the following three cases:

\underline{Case 1}: There are at least three elements in $[m]\setminus V$; without loss of generality, we may assume $1,2,3\in [m]\setminus V$. Then $d=|c|-2\ge c_1+c_2+c_3-2\ge 3k-2$. Therefore $k=\left\lceil \tfrac1 2d\right\rceil\ge\left\lceil \tfrac 1 2(3k-2)\right\rceil\ge \tfrac 3 2k-1$, so $k\le 2$. Therefore $k=2$, so the dimension of $\Bier_c(M)$ is either $3$ or $4$ and $c_1,c_2,c_3\geq 2$. Therefore, we get $c=(2,2,2)$, and there are no ghost vertices in $\Bier_c(M)$. If $a+b+c\le k=2$ then by \Cref{NeighborlySpheresThm} we have $x_1^ax_2^bx_3^c\in M$ and $x_1^{2-a}x_2^{2-b}x_3^{2-c}\not\in M$, hence $M_3^2\subseteq M\subseteq M_3^3$, hence we get the case (d).

\underline{Case 2}: There exist exactly two elements in $[m]\setminus V$; we may assume that $[m]\setminus V=\{1,2\}$. Since $k=\left\lceil \tfrac 1 2d\right\rceil$ we have $d=2k, 2k-1$ and $2k+2\ge d+2=|c|\ge 2k+1$. Hence, by Theorem~\ref{NeighborlySpheresThm}, for each $k\ge 2$, we have the following possibilities for $c$: 
\begin{itemize}
    \item $(k+2,k)$, $(k+1,k+1)$, $(k+1,k,1)$, $(k,k,2)$, $(k,k,1,1)$, and $M_2^k\subseteq M\subseteq M_2^{k+1}$ (in the first two cases the Murai sphere has no ghost vertices, only one ghost vertex $x_{3,0}$ in the next two cases, and exactly two ghost vertices in the last case), which is the case (e);
    \item $(k+1,k)$, $(k,k,1)$, and $M=M_2^k$ (the Murai sphere has no ghost vertices in the first case and one ghost vertex in the second case), which is the case (f).
\end{itemize}

\underline{Case 3}: There exists exactly one element in $[m]\setminus V$; we may assume that $[m]\setminus V=\{1\}$. Then $|c|=2k+1, 2k+2$ and $c_1\geq k$. By Theorem~\ref{NeighborlySpheresThm} we get $M=\langle x_{1}^a\rangle$, where $k\le a\le |c|-k-1 = k, k+1$, which is the case (g). 
\end{proof}

\section{Simplicial multiwedge construction}

In this section, we discuss a classification of simplicial $(d-1)$-spheres with $d+3$ vertices for $d\geq 2$ by means of applying the simplicial multiwedge operation to the boundaries of cyclic polytopes. In these terms, we prove necessary and sufficient conditions for such a sphere to be neighborly. 

We start by recalling the definition of a cyclic polytope. 

\begin{definition}
Let $n\geq 2, m\geq n+1$ be integers and $T=\{t_1,\ldots,t_m\}$ with $t_1<t_2<\ldots<t_m$ be a set of real numbers. Consider the moment map: 
$$
\mu\colon\R\to\R^n,\,t\mapsto (t,t^2,\ldots,t^n).
$$    
Then a \emph{cyclic polytope} $C(m,n)$ is defined as the convex hull of the set of points $\{\mu(t_1),\mu(t_2),\ldots,\mu(t_m)\}$ in $\R^n$. Its boundary is denoted by $\Delta(m,n)$. 
\end{definition}

It is well-known that $C(m,n)$ is a simplicial $n$-polytope with $m$ vertices $\{\mu(t_i)\mid 1\leq i\leq m\}$ and its combinatorial type does not depend on the choice of the set of distinct points $T$ in $\R^n$. Furthermore, $\Delta(m,n)$ is a neighborly polytopal $(n-1)$-sphere for any $n\geq 2$ and $m\geq n+1$. 

Recall that the \emph{Gale evenness condition} provides a way to determine a facet of a cyclic polytope as follows. A subset $T_{n}\subseteq T$ of cardinality $n$ forms a facet of $C(m,n)$ if and only if any two elements in its complement $T\setminus T_{n}$ are separated by an even number of elements from $T_{n}$ in the sequence  $(t_{1},t_{2},\ldots ,t_{m})$. 

Due to~\cite[Proposition 1]{Mani}, all simplicial $(d-1)$-spheres with $d+3$ vertices are polytopal. Moreover, they can be classified by means of the following simplicial operation, which plays a crucial role in this paper.

\begin{construction}[\cite{BBCG}]
Let $m\geq 2$ and $K$ be a simplicial complex on $[m]$. Suppose $J=(j_1,\ldots,j_m)\in\N^m$ is an $m$-tuple of positive integers. For each monomial 
$$
v:=x_{i_1}\cdots x_{i_p}
$$ 
in the minimal set of generators $G(I_K)$ of the Stanley--Reisner ideal $I_K\subset\Q[m]$ of $K$, where
$$
\Q[m]:=\Q[x_1,\ldots,x_m],
$$
we define the corresponding monomial 
$$
v(J):=x_{i_1,1}\cdots x_{i_1,j_{i_1}}\cdots x_{i_p,1}\cdots x_{i_p,j_{i_p}}
$$
in the polynomial algebra with $|J|$ generators:
$$
\Q[m](J):=\Q[x_{1,1}\ldots,x_{1,j_1},\ldots,x_{m,1},\ldots,x_{m,j_m}].
$$
Then the result of the \emph{simplicial multiwedge} (or, \emph{$J$-construction}) applied to $K$ is the simplicial complex $K(J)$ on $[|J|]$ such that
$$
\Q[K(J)]:=\Q[m](J)/I_K(J), \text{ where }I_{K}(J):=(v(J)\mid v\in G(I_K)).
$$
\end{construction}

\begin{remark}
Observe that for $J=(1,1,\ldots,1)$, the resulting simplicial complex $K(J)$ is isomorphic to $K$. In general, both the dimension and the number of vertices of $K(J)$ increase by $|J|-m$ compared to those of $K$. Therefore, this operation preserves the polytopality of a sphere, the difference $m-n$, and the cardinality of the minimal set of generators $G(I_K)$ of the Stanley--Reisner ideal $I_K$ of $K$. 
\end{remark}

\begin{example}
It follows from the above definition that any product of $r$ simplices is the result of a simplicial multiwedge operation applied to the cube $I^r$ for any $r\geq 1$. In particular, the next isomorphisms hold:
$$
\Delta(n+1,n)(J)=\partial \Delta_{[|J|]} = \Delta(|J|,|J|-1), \text{ for each }J\in\N^{n+1}.
$$    
\end{example}

Now we are ready to state the classification theorem for combinatorial simple polytopes with the number of facets equal $m=n+3$ the way it was formulated in~\cite[Theorem 2.3.48]{Er}.

\begin{theorem}[\cite{Er}]\label{ErokhovetsClassificationThm}
Any simple polytope with $m=n+3, n\geq 2$ is isomorphic to a product of three simplices, or to the result of a simplicial multiwedge operation applied to an even-dimensional dual cyclic polytope $C=C^*(2p+3,2p)$ with $p\geq 1$.
\end{theorem} 

Due to the result of~\cite{Sch}, the number of elements in $G(I_\Delta)$ equals $\beta_1(\Q[\Delta])=2p+3$, where $\Delta=\Delta(2p+3,2p)$. In what follows, we will need a refinement of this result: namely, we will be interested in the number of monomials of degree $q\geq 2$ in $G(I_K)$. There is the following description of the Stanley-Reisner ideal for the boundary of an even-dimensional cyclic polytope, see~\cite[Lemma 3.1]{BohmPap}. Suppose $a, r$, and $s$ are positive integers with $r<s$ and $2a\leq s-r+2$. Define the ideal $I_{a,r,s}\subset F[x_{r},\ldots,x_{s}]$ by
\[
I_{a,r,s}:=(x_{t_1}\ldots x_{t_a}\mid r\leq t_1,t_a\leq s,t_{j}+2\leq t_{j+1}\text{ for }1\leq j\leq a-1).
\]
Then the Stanley-Reisner ideal of a dual cyclic polytope is described as follows:
\[
I_{SR}(\Delta(m,d))=(I_{a,1,m-1},I_{a,2,m}),
\]
where $d$ is even, $2\leq d<m-1$, and $a=\tfrac 1 2(d+2)$.

Our goal in this section and in the next one is as follows: for each of the spheres from Theorem~\ref{NeighborlyExplicitClassificationThm} (e), we are going to find a positive integer $p$ and a $(2p+3)$-tuple of positive integers $J$ such that our neighborly Murai sphere is isomorphic to the nerve complex $\Delta(p,J)$ of the simple polytope $C^*(2p+3,2p)(J)$.

We start with the next proposition, which provides a useful criterion for the sphere $\Delta(p,J)$ to be neighborly. In order to state it, we introduce the following notation. 

Suppose $t:=(t_1,t_2,\ldots,t_{p+1})$ is a tuple of positive integers satisfying one of the two conditions: either
\[
1\leq t_1<t_1+2\leq t_2<t_2+2\leq\ldots<t_p+2\leq t_{p+1}\leq 2p+2
\]
or
\[
2\leq t_1<t_1+2\leq t_2<t_2+2\leq\ldots<t_p+2\leq t_{p+1}\leq 2p+3.
\]
In what follows, we call such a $(p+1)$-tuple $t$ \emph{admissible}. For a $(p+1)$-tuple $t$ we define the positive integer vectors 
\[
J(t):=(j_{t_1},\ldots,j_{t_{p+1}})\in\ZZ^{p+1}\text{ and }J(t^c):=(j_k\mid 1\leq k\leq 2p+3,k\neq t_1,t_2,\ldots,t_{p+1})\in\ZZ^{p+2}.
\] 

\begin{proposition}\label{NeighborCriterionProp}
The sphere $\Delta(p,J)$ is neighborly if and only if $|J(t)|+2\geq |J(t^c)|$ for each admissible tuple.
\end{proposition} 
\begin{proof}
First, note that the simplicial multiwedge construction preserves the difference between the number of vertices and the dimension of a simplicial complex. Since $\Delta(2p+3,2p)=\Delta(p,(1,1,\ldots,1))$, it follows that 
$$
\dim(\Delta(p,J))=\dim(\Delta(2p+3,2p))+|J|-(2p+3)=|J|-4.
$$
Hence, $\Delta(p,J)$ is neighborly if and only if each set $F$ of its vertices with $|F|\leq\left\lceil\frac{|J|}{2}\right\rceil-2$ forms a face. 

It follows from the Hochster theorem~\cite{Hoh} that for any simplicial complex $K$, the number of minimal non-faces of cardinality $q$ equals the bigraded Betti number 
$$
\beta^{-1,2q}(K):=\beta^{-1,2q}(\Q[K])=\sum\limits_{I\subseteq [m], |I|=q}\dim_{\Q}\tilde{H}^{q-2}(K_I).
$$
On the other hand, by~\cite[Corollary 7.9]{Ayz13}, we have an identity for the bigraded Betti numbers of $\Delta(p,J)$:
$$
\beta^{-1,2q}(\Delta(p,J))=\sum\beta^{-1,2\bf{a}}(\Delta(2p+3,2p)) \text{ for all vectors }{\bf{a}}\in\{0,1\}^{2p+3}\text{ such that }a_1j_1+\ldots+a_{2p+3}j_{2p+3}=q.
$$

Moreover, for any simplicial complex $K$, the multigraded Betti number $\beta^{-1,2\bf{a}}(K)$ equals $1$, if $\bf{a}$ represents a minimal non-face of $K$, and equals $0$, otherwise. Therefore, $\Delta(p,J)$ is a neighborly sphere if and only if 
\[
\beta^{-1,2q}(\Delta(p,J))=0\text{ for all }q\leq \left\lceil\tfrac 1 2|J|\right\rceil-2,
\]
which is equivalent to $|J(a)|=a_{1}j_1+\ldots+a_{2p+3}j_{2p+3}>\left\lceil\tfrac1 2|J|\right\rceil-2$ for all ${\bf{a}}\in\min(\Delta(2p+3,2p))$. On the other hand, the minimal non-faces are described via the Stanley-Reisner ideal generators as 
$$
I_{SR}(\Delta(2p+3,2p))=(x_t\mid t\text{ is admissible}).
$$

It remains to observe that $|J|=|J(t)|+|J(t^c)|$ for each admissible tuple $t\in\N^{p+1}$, and hence $|J(t)|+2>\left\lceil\tfrac1 2|J|\right\rceil$ if and only if $|J(t)|+2\geq \tfrac 1 2(|J|+2)$ if and only if $|J(t)|+2\geq |J(t^c)|$ for each admissible tuple $t\in\N^{p+1}$. This finishes the proof. 
\end{proof}

\begin{lemma}\label{OneTwoVectorLemma}
Suppose $p\in \N$ and $J\in\N^{2p+3}$. If $\Delta(p,J)$ is neighborly, then $J\in\{1,2\}^{2p+3}$.  
\end{lemma}
\begin{proof}
Let us use the previous criterion for neighborness of the sphere $\Delta(p,J)$.
Firstly, we show that $j_{2k+1}\leq 2$ for any $0\leq k\leq p+1$. Consider the admissible tuple $t=(1,3,\ldots,2k-1,2k+2,2k+4,\ldots,2p+2)$. It yields the following inequality:
$$
j_1+j_3+\ldots+j_{2k-1}+j_{2k+2}+\ldots+j_{2p+2}+2\geq j_2+j_4+\ldots+j_{2k}+j_{2k+1}+\ldots+j_{2p+3}.
$$
Then consider the admissible tuple $t=(2,4,\ldots,2k,2k+3,2k+5,\ldots,2p+3)$. It yields the following inequality:
$$
j_2+j_4+\ldots+j_{2k}+j_{2k+3}+\ldots+j_{2p+3}+2\geq j_1+j_3+\ldots+j_{2k+1}+j_{2k+2}+\ldots+j_{2p+2}.
$$
Adding these two inequalities gives $|J|-j_{2k+1}+4\geq |J|+j_{2k+1}$, which implies the desired inequality $j_{2k+1}\leq 2$.

Finally, we show that $j_{2k+2}\leq 2$ for any $0\leq k\leq p$. Consider the admissible tuple $t=(2,4,\ldots,2k,2k+3,2k+5,\ldots,2p+3)$. It yields the following inequality:
$$
j_2+j_4+\ldots+j_{2k}+j_{2k+3}+\ldots+j_{2p+3}+2\geq j_1+j_3+\ldots+j_{2k+1}+j_{2k+2}+\ldots+j_{2p+2}.
$$
Then consider the admissible tuple $t=(1,3,\ldots,2k+1,2k+4,2k+6,\ldots,2p+2)$. It yields the following inequality:
$$
j_1+j_3+\ldots+j_{2k+1}+j_{2k+4}+\ldots+j_{2p+2}+2\geq j_2+j_4+\ldots+j_{2k}+j_{2k+2}+j_{2k+3}+\ldots+j_{2p+3}.
$$
Adding these two inequalities gives $|J|-j_{2k+2}+4\geq |J|+j_{2k+2}$, which implies the desired inequality $j_{2k+2}\leq 2$.
\end{proof}

\begin{remark}
The previous statement shows that for each positive integer $p$ there are only finitely many $(2p+3)$-tuples of positive integers $J$ such that the sphere $\Delta(p,J)$ is neighborly. It also immediately implies that the spheres $\Delta(p,(1,\ldots,1,2,1,\ldots,1))$ and $\Delta(p,(2,\ldots,2))$ are neighborly for all $p\geq 1$.
\end{remark}

\begin{lemma}\label{NumThLemma}
Let $p\geq 0$ and $S$ be a sequence of ‘0's and ‘1's of length $2p+1$ viewed in a cyclic order. Then the following conditions are equivalent:
\begin{itemize}
\item ``Inequality condition'': inserting one extra ‘1' between any two consecutive elements and considering this position as even makes the sum of elements of $S\cup\{1\}$ in even positions be greater or equal to the one in odd positions;
\item ``Evenness condition'': between any two consecutive ‘1's in $S$ there is an even number of ‘0's.
\end{itemize}
\end{lemma}
\begin{proof}
Induction on $p$. Base case: $p=0$. If $S=\{0\}$, then Inequality condition holds as $1\geq 0$ and Evenness condition holds as there are no ‘1's in $S$. If $S=\{1\}$, then Inequality condition holds as $1\geq 1$ and Evenness condition holds as there are no ‘0's in $S$ and zero is an even number. If $p=1$, the only triple $S$ not satisfying the Evenness condition has one ‘0' and two ‘1's and is the only triple not satisfying the Inequality condition.

Suppose the statement holds when the length of a sequence is less or equal to $2p+1$ and let $S$ be a sequence of ‘0's and ‘1's of length $2p+3$ and $p\geq 1$.

Assume that the Evenness condition holds for $S$. If there are no ‘0's in $S$, then each inequality in the Inequality condition turns into equality $p+2=p+2$, which is true. If there are ‘0's
in $S$, then Evenness condition implies there exist two consecutive ‘0's in $S$. Observe that deleting them from $S$ results in a sequence $S'$ of ‘0's and ‘1's of length $2p+1$ satisfying the Evenness condition. Hence, by the inductive assumption, $S'$ satisfies the Inequality condition. It is easy to see that it follows $S$ also satisfies the Inequality condition: indeed, if the extra ‘1' is not between the two consecutive ‘0's that we deleted from $S$, then those zeros are on both sides of the corresponding inequality and it still holds; if the extra ‘1' is between those two consecutive ‘0's, then from the inductive assumption for the sequence $\{s_1,\ldots,s_{2p+1}\}$ ($s_0=s_{2p+2}=0$) we get:
$$
1+s_3+s_5+\ldots+s_{2p+1}\geq s_1+s_2+s_4+\ldots+s_{2p}
$$
and 
$$
1+s_1+s_3+\ldots+s_{2p-1}\geq s_{2p+1}+s_2+s_4+\ldots+s_{2p},
$$
while we must show that
$$
1+s_1+s_3+\ldots+s_{2p+1}\geq s_{2}+s_4+\ldots+s_{2p}.
$$
Therefore, adding the first two inequalities together yields:
$$
1+s_3+s_{5}+\ldots+s_{2p-1}\geq s_2+s_4+\ldots+s_{2p},
$$
which implies the desired inequality, since $s_1, s_{2p+1}\geq 0$.

Assume that the Inequality condition holds for $S$. Then either there are two consecutive ‘0's, or two consecutive ‘1's in $S$: otherwise, $S$ has the form $\{0,1,0,1,0,\ldots\}$ and does not satisfy the Inequality condition: $1=1+0+0+\ldots+0 < 1+\ldots+1=p+1$ or $p+2$.
If there are two consecutive ‘0's in $S$, we delete them and get the sequence $S'$ of ‘0's and ‘1's, which satisfies the Inequality condition. Therefore, by the inductive assumption, $S'$ satisfies the Evenness condition and hence $S$ also satisfies the Evenness condition. If there are two consecutive ‘1's in $S$, we delete them and get the sequence $S'$ of ‘0's and ‘1's, which satisfies the Inequality condition. Therefore, by the inductive assumption, $S'$ satisfies the Evenness condition. Hence either $S$ also satisfies the Evenness condition, or there are two intervals of ‘0's in $S$, both having odd lengths and separated by two consecutive ‘1's. This yields either a circle
$$
\{1,0,\ldots,0,1,1,0,\ldots,0\},
$$
or a circle 
$$
\{1,0,\ldots,0,1,1,0,\ldots,0,1,S''\},
$$
where $S''$ has odd length and satisfies the Evenness condition. 

In the first case, we place the extra ‘1' between our two separating ‘1's and get the inequality $1=1+0+\ldots+0\geq 1+1+1+0+\ldots+0=3$, which is false.

In the second case, denote by $S_1$ and $S_2$ the sums of the elements of $S''$ in odd and even positions, respectively. Since $S$ satisfies the Inequality condition, we get the following two inequalities:
$$
1+S_1\geq S_2+4 \quad\text{(placing the extra ‘1' between our two separating ‘1's)
}
$$
and
$$
1+S_2+2\geq S_1+2\quad\text{(placing the extra ‘1' between $S''$ and $1$).}
$$
Adding these inequalities together yields $4\geq 6$, which is false. Thus, this case is impossible and $S$ satisfies the Evenness condition, which finishes the proof.
\end{proof}

In what follows we will need the next simple observation.

\begin{proposition}\label{NeighborlyProductsOfSimplicesProp}
Suppose a triangulated sphere is isomorphic to a join of boundaries of simplices:
$$
S\cong \partial\Delta^{n_1}\ast\ldots\ast\partial\Delta^{n_k},\text{ where } n_1\geq n_2\geq\ldots\geq n_k\geq 1\text{ and }k\geq 1.
$$
Then $S$ is neighborly if and only if one of the following conditions holds:
\begin{itemize}
\item $k=1$;
\item $k=2$ and $n_1-1\leq n_2\leq n_1$;
\item $k=3$ and $n_1=n_2=n_3=1$.
\end{itemize}
In particular, every neighborly join of boundaries of simplices is isomorphic to a Murai sphere. 
\end{proposition}
\begin{proof}
By definition of join, the dimension of $S$ equals $d:=(n_1-1)+\ldots+(n_k-1)+(k-1)=n_1+\ldots+n_k-1\geq kn_k-1$. By definition of neighborness, $S$ is neighborly if and only if $n_k+1>\left\lceil \tfrac 1 2d\right\rceil$, which is equivalent to $n_k\geq \tfrac 1 2d$.
It follows that $2n_k\geq d\geq kn_k-1$, which implies that $k\le 3$.

When $k=1$, we obtain an arbitrary simplex $S\cong\partial\Delta^{n_1}$. When $k=2$, neighborness of $S\cong\partial\Delta^{n_1}\ast\partial\Delta^{n_2}$ is equivalent to $2n_2\geq n_1+n_2-1$, or $n_1-1\leq n_2\leq n_1$. When $k=3$, neighborness of $S\cong\partial\Delta^{n_1}\ast\partial\Delta^{n_2}\ast\partial\Delta^{1}$ is equivalent to $2=2n_3\geq n_1+n_2+1-1$, which is equivalent to $n_1=n_2=n_3=1$.
\end{proof}

From now on, by a vector $J\in\N^m, m\geq 2$, satisfying the \emph{Evenness condition}, we mean a $\{1,2\}$-vector such that between any two consecutive '2's there is an even number of '1's, when the components of $J$ are viewed in cyclic order.

\begin{theorem}\label{NeigborlyJ}
Let $d\geq 2$. Then $K$ is a neighborly $(d-1)$-sphere with $d+3$ vertices if and only if one of the following two conditions holds:
\begin{itemize}
\item $K\cong\partial\Delta^1\ast\partial\Delta^1\ast\partial\Delta^1$;
\item there exist $p\in\N$ and $J\in\N^{2p+3}$ with $|J|=d+3$ satisfying the Evenness condition such that 
$K$ is isomorphic to $\Delta(p,J)$.
\end{itemize}
\end{theorem}
\begin{proof}
Suppose $d=2$. Then $K$ is a neighborly 1-sphere with 5 vertices if and only if $K$ is the boundary of a pentagon. Since $p\geq 1$ and $|J|=5$, we get $2p+3=|J|=5$ and the only such $J$ satisfying the Evenness condition is the 5-tuple $(1,1,1,1,1)$, which finishes the proof in this case.

Suppose $d\geq 3$ and $K$ is a neighborly $(d-1)$-sphere with $d+3$ vertices different from $\partial\Delta^1\ast\partial\Delta^1\ast\partial\Delta^1$. Due to~\cite{Mani}, the simplicial sphere $K$ is isomorphic to the boundary of a simplicial $d$-polytope with $d+3$ vertices. By Theorem~\ref{ErokhovetsClassificationThm}, the letter is isomorphic to a simplicial multiwedge over a cyclic polytope $C(2p+3,2p)$ for a certain $p\in\N$, since the neighborly sphere $K$ can not be a join of 3 boundaries of simplices due to Proposition~\ref{NeighborlyProductsOfSimplicesProp}. Therefore, $K$ is isomorphic to $\Delta(p,J)$ with $J\in\N^{2p+3}$. By Lemma~\ref{OneTwoVectorLemma}, $J$ is a $\{1,2\}$-vector. It remains to show that between any two consecutive ‘2's there are an even number of ‘1's when components of $J$ are viewed as forming a circle. 

Since $\Delta(p,J)$ is neighborly, by Proposition~\ref{NeighborCriterionProp} for each admissible $(p+1)$-tuple $t$ one has: $|J(t)|+2\geq |J(t^c)|$. It is easy to see that any admissible tuple has one of the following two forms:
\begin{itemize}
\item $t=(1,3,5,\ldots,2k-1,2k+2,2k+4,\ldots,2p+2)$ for some $k\in\{1,\ldots,p+1\}$;
\item $t=(2,4,6,\ldots,2k,2k+3,2k+5,\ldots,2p+3)$ for some $k\in\{1,\ldots,p+1\}$.
\end{itemize}

Then the inequality $|J(t)|+2\geq |J(t^c)|$ can be interpreted as follows: in each of the two cases above there is a unique ‘gap' of length two in the components of the admissible tuple $t$ when components of $J$ are viewed as forming a circle. Therefore, in what follows we consider the circular order on all tuples involved. Let us add to $J$ an extra component with value equal to $2$ and insert it between the elements of the ‘gap' in $t$. This yields a new $\{1,2\}$-vector $J'$ of even length $2p+4$. Moreover, we add this extra component with value $2$ to $t$, which gives rise to a tuple $t'$ of length $p+2$ in $J'$ such that between any two consecutive elements of $t'$ there is a unique element of $t^c$. Now, our inequality is equivalent to 
$$
|J'(t')| = |J(t)|+2\geq |J(t^c)| = |J'((t')^c)|.
$$

Thus, if we subtract $1$ from each component of $J'$, we get the Inequality condition of Lemma~\ref{NumThLemma}. Then by Lemma~\ref{NumThLemma}, $J$ satisfies the Evenness condition.

Suppose $d\geq 2$ and $K$ is isomorphic to $\Delta(p,J)$, where $J$ satisfies the evenness condition and $|J|=d+3$. Similarly to the above argument, using Lemma~\ref{NumThLemma}, we obtain that $J$ satisfies the inequality $2+|J(t)|\geq |J(t^c)|$ for each admissible tuple $t$. This implies that $\Delta(p,J)$ is neighborly, due to Proposition~\ref{NeighborCriterionProp}.
\end{proof}

If $m$ is odd, then it is well-known, see~\cite{Bjorner93}, that a simplicial $(n-1)$-dimensional sphere with $m=n+3$ vertices is neighborly if and only if it is the boundary of an even-dimensional cyclic polytope. Hence, to complete the classification of Murai spheres arising in Theorem~\ref{NeighborlyExplicitClassificationThm} (e), it remains to consider the case where the polytopal sphere $\Delta(p,J)$ has an even number of vertices.

%%%%%%%%%%%%%%%%%%%%%%%%%%%%%%%%%%%%%%%%%%%%%%%%%%%%%%%%%%%%%%%

\section{Neighborly Murai $(d-1)$-spheres with $d+3$ vertices}

In this section, we finish the classification of the Murai spheres identified in Theorem~\ref{NeighborlyExplicitClassificationThm} (e) and (f). As an application of our results, we show that any neighborly simplicial $(d-1)$-sphere with $\leq d+3$ vertices is isomorphic to a Murai sphere for each $d\geq 2$. 

By \Cref{NeighborlyExplicitClassificationThm}, we have 7 types of multicomplexes which induce a neigborly Murai $d$-sphere with $d+4$ vertices. Five of them induce an even-dimensional sphere (case (e)), and two of them induce an odd-dimensional sphere (case (f)). 

Let us first consider the case of odd dimensions. For every $k\in\NN$ we have two possibilities for $c$, namely $(k+1,k)$ and $(k,k,1)$, and in both cases only one $c$-multicomplex $M=M_2^k$. We find that the $(2k-1)$-dimensional neigborly Murai spheres $\Bier_{(k+1,k)}(M_2^k)$ and $\Bier_{(k,k,1)}(M_2^k)$ are isomorphic. An isomorphism is given by $\gen x1{k+1}\mapsto \gen x3{1}$ and $\gen xi{j}\mapsto \gen xi{j}$, for $(i,j)\ne(1,k+1)$. Note that $\gen x 3 0$ is a ghost vertex of $\Bier_{(k,k,1)}(M_2^k)$.

In the even dimensional case for every $k\in\NN$ we have five possibilities for $c$, namely $(k+2,k)$, $(k+1,k+1)$, $(k+1,k,1)$, $(k,k,2)$, or $(k,k,1,1)$. For every such $c$, let $\mathcal{M}^c$ be the set of isomorphic classes of neigborly Murai spheres of the form $\Bier_c(M)$, where $M_2^k\subseteq M\subseteq M_2^{k+1}$. By \Cref{NeigborlyJ}, a sphere $\Delta(p,J)$ is neigborly if and only if $J\in J_k$, where $\mathcal{J}_k$ is the set of all sequences $J\in\{1,2\}^{2p+3}$ with at least one 2 such that between any two consecutive 2's in $J$ there are evenly many 1's when components in $J$ are viewed as forming a circle and $\sum_{i=0}^{2p+3} j_i=2k+4$. Let $\mathcal{M}_k^J$ be the set of isomorphic classes of neighborly spheres of the form $\Delta(p,J)$, where $J\in \mathcal{J}_k$. In this section, we will prove the following theorem:

\begin{theorem}\label{MuraiClassesCoincideThm}
For every $k\in\NN$ the sets $\mathcal{M}^{(k+2,k)}$, $\mathcal{M}^{(k+1,k+1)}$, $\mathcal{M}^{(k+1,k,1)}$, $\mathcal{M}^{(k,k,2)}$, $\mathcal{M}^{(k,k,1,1)}$, and $\mathcal{M}_k^J$ coinside.
\end{theorem}

Let $k\in\NN$ and let $c=(c_1,c_2,\ldots)$ be $(k+2,k)$, $(k+1,k+1)$, $(k+1,k,1)$, $(k,k,2)$, or $(k,k,1,1)$. Let $M$ be a $c$-multicomplex such that $M_2^k\subseteq M\subseteq M_2^{k+1}$. Let $\{x_1^{\alpha_1}x_2^{\beta_1},\ldots, x_1^{\alpha_r}x_2^{\beta_r}\}$ be the minimal set of its generators, then $\alpha_i\ne \alpha_j$ and $\beta_i\ne \beta_j$ for $i\ne j$. So we may assume that $0\le \alpha_1<\alpha_2<\ldots<\alpha_r$ and $\beta_1>\beta_2>\ldots >\beta_r\ge 0$, and because of the assumption we have $k\le \alpha_i+\beta_i\le k+1$ for all $i$.

Suppose that $\alpha_i+\beta_i=k+1$. If $\alpha_{i+1}=\alpha_i+1$ then by the assumption we have $\beta_{i+1}=\beta_i-1$ or $\beta_{i+1}=\beta_i-2$. If $\alpha_{i+1}=\alpha_i+2$ and $\beta_{i+1}=\beta_i-3$ then $x_1^{\alpha_i+1}x_2^{\beta_i-2}\not\in M$ but this is not possible by the assumption, since $(\alpha_i+1)+(\beta_i-2)=k$. So if $\alpha_{i+1}=\alpha_i+2$ then $\beta_{i+1}=\beta_i-2$.
For $\alpha_{i+1}=\alpha_i+p$, where $p\ge 3$, we have $\beta_{i+1}\le \beta_i-p$, since $\alpha_{i+1}+\beta_{i+1}\le k+1$. But then $x_1^{\alpha_i+1}x_2^{\beta_i-2}\not\in M$, which is not possible.

Suppose that $\alpha_i+\beta_i=k$. If $\alpha_{i+1}=\alpha_i+1$ then by the assumption we have $\beta_{i+1}=\beta_i-1$. If $\alpha_{i+1}=\alpha_i+2$ and $\beta_{i+1}=\beta_i-2$ then $x_1^{\alpha_i+1}x_2^{\beta_i-1}\not\in M$, which is not possible, so $\beta_{i+1}=\beta_i-1$. As in the previous case, we see that $\alpha_{i+1}> \alpha_i+2$ is not possible.

\begin{lemma}\label{lemma_minimalnon}
Let $k\in\NN$, let $c$ denote one of $(k+2,k)$, $(k+1,k+1)$, $(k+1,k,1)$, $(k,k,2)$, $(k,k,1,1)$, and let $M=\langle x_1^{\alpha_1}x_2^{\beta_1},\ldots, x_1^{\alpha_r}x_2^{\beta_r}\rangle$, where $0\le \alpha_1<\alpha_2<\ldots<\alpha_r$ and $\beta_1>\beta_2>\ldots >\beta_r\ge 0$.
For every $i=1,\ldots,r-1$ there exists exactly one minimal non-element $x_1^{\gamma_i}x_2^{\delta_i}\in\min M$, such that $\alpha_i<\gamma_i\le \alpha_{i+1}$. We have $(\gamma_i,\delta_i)=(\alpha_i+1,\beta_{i+1}+1)$, so
\[
\gamma_i+\delta_i=\begin{cases}
        k+2,& \alpha_i+\beta_i=k+1,\ \alpha_{i+1}=\alpha_i+1\text{ and},\ \beta_{i+1}=\beta_i-1,\\
        k+1,&\text{ otherwise},
    \end{cases}
\]    
for all $i=1,\ldots, r-1$.
\end{lemma}

\begin{proof}
We have to consider five cases:
\begin{itemize}
\item  Let $\alpha_i+\beta_i=k$ and $(\alpha_{i+1},\beta_{i+1})=(\alpha_i+1,\beta_i-1)$. Then $\gamma_i=\alpha_{i+1}$ and the only minimal non-element of the form $x_1^{\gamma_i}x_2^\delta$ is $x_1^{\alpha_i+1}x_2^{\beta_i}\in\min M$. Hence $\delta_i=\beta_i$ and therefore $\gamma_i+\delta_i=k+1$.
\item  Let $\alpha_i+\beta_i=k$ and $(\alpha_{i+1},\beta_{i+1})=(\alpha_i+2,\beta_i-1)$. Suppose there is $\beta$ such that $x_1^{\alpha_i+2}x_2^\beta\in\min(M)$. Since $x_1^{\alpha_{i+1}}x_2^{\beta_{i+1}}=x_1^{\alpha_i+2}x_2^{\beta_i-1}\in M$ we have $\beta\ge \beta_i$. But for such $\beta$ we have $x_1^{\alpha_i+1}x_2^\beta\not\in M$, so $\gamma_i=\alpha_i+1$. Since $x_1^{\gamma_i}x_2^{\beta_i-1}\in M$ we get $\delta_i=\beta_i$ and $\gamma_i+\delta_i=k+1$.
\item  Let $\alpha_i+\beta_i=k+1$ and $(\alpha_{i+1},\beta_{i+1})=(\alpha_i+1,\beta_i-1)$. Then $\gamma_i=\alpha_i+1$, $\delta_i=\beta_i$, and $\gamma_i+\delta_i=k+2$.
\item  Let $\alpha_i+\beta_i=k+1$ and $(\alpha_{i+1},\beta_{i+1})=(\alpha_i+1,\beta_i-2)$. Then $\gamma_i=\alpha_i+1$, $d_i=\beta_i-1$, and $\gamma_i+\delta_i=k+1$.
\item  Let $\alpha_i+\beta_i=k+1$ and $(\alpha_{i+1},\beta_{i+1})=(\alpha_i+2,\beta_i-2)$. As in the second case, we see that $\gamma_i=\alpha_i+1$. Then $\delta_i=\beta_i-1$ and $\gamma_i+\delta_i=k+1$. \qedhere
\end{itemize}    
\end{proof}

Let $M=\langle x_1^{\alpha_1}x_2^{\beta_1},\ldots, x_1^{\alpha_r}x_2^{\beta_r}\rangle$ be as in the previous lemma. For every $i=1,\ldots,r$ the element $\sigma_{2i+1}=x_{1,\alpha_{i+1}+1}\ldots x_{1,c_1}x_{2,\beta_{i+1}+1}\ldots x_{2,c_2}$ is a generator of Stanley--Reisner ideal $I_{SR}(\Bier_c(M))$. By the previous lemma for every $i=1,\ldots,r-1$ the element $\sigma_{2i}=x_{1,0}\ldots x_{1,\gamma_i-1}x_{2,0}\ldots x_{2,\delta_i-1}$ is a generator of Stanley--Reisner ideal $I_{SR}(\Bier_c(M))$. 

From the previous lemma the generator $\sigma_{2i}=x_{1,0}\ldots x_{1,\gamma_i-1}x_{2,0}\ldots x_{2,\delta_i-1}$ is disjoint with the generators $\sigma_{2i-1}=x_{1,\alpha_i+1}\ldots x_{1,c_1}x_{2,\beta_i+1}\ldots x_{2,c_2}$ and $\sigma_{2i+1}=x_{1,\alpha_{i+1}+1}\ldots x_{1,c_1}x_{2,\beta_{i+1}+1}\ldots x_{2,c_2}$ of Stanley--Reisner ideal $I_{SR}(\Bier_c(M))$. Also for every $i=0,\ldots,2r$, we have $j_i=|V(\Bier_c(M))|-|\sigma_i|-|\sigma_{i+1}|$ is 1 or 2 in particular it is not possible that $|\sigma_i|=|\sigma_{i+1}|=k+2$ for some $i$. Suppose that $j_i=2$ and $j_{i+1}=1$. Then $|\sigma_i|=|\sigma_{i+1}|=k+1$ and $\sigma_{i+2}=k$, so $\sigma_{i+3}=k+1$ and $j_{i+2}=1$. In a similar way, we can see that $j_{i+3}=2$ or $(j_{i+3},i_{i+4})=(1,1)$, so the smallest $p$ such that $2=j_{i+p}$ is an odd number, i.e. there are evenly many 1's between $2=j_i$ and the next $2=j_{i+p}$.

For simplicity, until the end of this section we will use the notation $\Delta(2p+3,J):=\Delta(p,J), p\geq 1$, so that $J$ has $2p+3$ components. We will show that every Murai sphere from the previous Lemma is isomorphic to $\Delta(2r+1,J)$ or $\Delta(2r-1,J)$ for some $J$ and $r$ to be the minimal number of generators of $M$. Let us recall the definition of the complex $\Delta(2r+1,J)$, where $r\in\NN$ and $J=(j_0,\ldots,j_{2r})\in\NN^{2r+1}$. In what follows, all indices will be calculated $\pmod{2r+1}$. We define the simplicial complex $\Delta(2r+1,J)$ with the set of vertices $V(\Delta(2r+1,J))=\{y_{p,s}\mid 0\le p\le 2r,1\le s\le j_p\}$ and the set of minimal nonsimplices $\min(\Delta(2r+1,J))=\{\tau_i\mid i=0,\ldots, 2r\}$, where $\tau_i=\{y_{p,s}\mid p\in J_i, 1\leq s\leq j_p\}$ and $J_i=\{i+2q \mid q=0,\ldots,r-1\}$. So $|\tau_i|=\sum_{q=0}^{r-1}j_{i+2q}=\sum_{p\in J_i}j_p$. 
If $J'=(j'_0,\ldots,j'_{2r})$ is such that $j'_i=j_{i+1}$ for all $i$ then $\Delta(2r+1,J)$ and $\Delta(2r+1,J')$ are isomorphic via isomorphism $y_{p,s}\mapsto y_{p+1,s}$. If $j'_i=j_{2r-i}$ then $\Delta(2r+1,J)$ and $\Delta(2r+1,J')$ are isomorphic via isomorphism $y_{p,s}\mapsto y_{2r-1,s}$. So the isomorphism type of $\Delta(2r+1,J)$ is independent of cyclic permutation and reversion of the sequence $J$.

\begin{example}
Let $J=(2,2,2,1,1)$. Then $\tau_0=\{y_{0,1},y_{0,2},y_{2,1}\}$, $\tau_1=\{y_{1,1},x_{1,2},y_{3,1},y_{3,2}\}$, $\tau_2=\{y_{2,1},y_{4,1}\}$, $\tau_3=\{y_{3,1},y_{3,2},y_{0,1},y_{0,2}\}$, $\tau_4=\{y_{4,1},y_{1,1},y_{1,2}\}$, and for every $i=0,\ldots, 4$ we have $\tau_i\cap\tau_{i+1}=\emptyset$.
Let us order the vertices by the lexicographical order of their indices, where the first numbers of the indices are ordered as $0<2<4<1<3$. If we write down all the vertices $3$ times ($3=r+1$) in order, we get
\[
\underbrace{y_{1,1}\!,\!y_{1,2}\!,\!y_{3,1}}_{\tau_1},\underbrace{y_{0,1}\!,\!y_{0,2}}_{M_0},\underbrace{y_{2,1}\!,\!y_{2,2}\!,\!y_{4,1}}_{\tau_2},\underbrace{y_{1,1}\!,\!y_{1,2}}_{M_1},\underbrace{y_{3,1}\!,\!y_{0,1}\!,\!y_{0,2}}_{\tau_3},\underbrace{y_{2,1}\!,\!y_{2,2}}_{M_2},\underbrace{y_{4,1}\!,\!y_{1,1}\!,\!y_{1,2}}_{\tau_4},\underbrace{y_{3,1}}_{M_3},\underbrace{y_{0,1}\!,\!y_{0,2}\!,\!y_{2,1}\!,\!y_{2,2}}_{\tau_0},\underbrace{y_{4,1}}_{M_4},
\]
where $M_{i}:=V\setminus(\tau_{i+1}\cup\tau_{i+2})=\{y_{i,s}\mid 1\le s\le j_{i}\}$ has $j_{i}$ elements.

For $c=(3,3)$ and $M=\langle x_2^3,x_1^2x_2\rangle$ the Murai sphere has five minimal nonsimplices: $\sigma_0=\left\{\gen x10,\gen x11,\gen x12\right\}$, $\sigma_1=\left\{\gen x13,\gen x22,\gen x23\right\}$, $\sigma_2=\left\{\gen x10,\gen x20,\gen x21\right\}$, $\sigma_3=\left\{\gen x11,\gen x12,\gen x13\right\}$, and $\sigma_4=\left\{\gen x20,\gen x21,\gen x22,\gen x23\right\}$. If we write all the vertices of the Murai sphere 3 times, we get 
\[
\underbrace{\gen x10,\gen x11,\gen x12}_{\sigma_0},\gen x13,\underbrace{\gen x23,\gen x22,\gen x21,\gen x20}_{\sigma_4},\gen x10,\underbrace{\gen x11,\gen x12,\gen x13}_{\sigma_3},\gen x23,\gen x22,\underbrace{\gen x21,\gen x20,\gen x10}_{\sigma_2},\gen x11,\gen x12,\underbrace{\gen x13,\gen x23,\gen x22}_{\sigma_1},\gen x21,\gen x20.
\]
Since $|\tau_i|=|\sigma_{4-i}|$, we get the same pattern as above, so the Murai sphere $\Bier_c(M)$ is isomorphic to $\Delta(5,J)$, and it is easy to read an isomorphism from both sequences. The isomorphism class of $\Delta(5,J)$ is independent of a cyclic permutation or reversion of $J$. So $\Bier_c(M)$ is isomorphic to $\Delta(5,J)$, where the components of $J$ are $j_i=|V(\Bier_c(M))|-|\sigma_i|-|\sigma_{i+1}|$.
\end{example}

\begin{corollary}
\begin{enumerate}
\item Let $J=(j_0,\ldots,j_{2r})\in\N^{2r+1}$. Then for the minimal non-simplices $\tau_i$, $i=0,\ldots,2r$
we have $\tau_i\cap \tau_{i+1}=\emptyset$ and $|V(\Delta(2r+1,J)|-|\tau_i|-|\tau_{i+1}|=j_{i-1}$ for all $i=1,\ldots,2r+1$.
\item Let $K$ be a simplicial complex with the set of minimal non-simplices $\min K=\{\sigma_0\ldots,\sigma_{2r}\}$. If for every $i=0,\ldots,2r$ we have $\sigma_i\cap\sigma_{i+1}=\emptyset$ and $|V(K)|-|\sigma_i|-|\sigma_{i+1}|=j_i$, then $K$ is isomorphic to $\Delta(2r+1,(j_0,\ldots,j_{2r}))$.
\end{enumerate}
\end{corollary}

\begin{corollary}
Let $J\in\N^{2r+1}$ and $J'\in\N^{2r'+1}$ be such that $|J|=|J'|$. Then $\Delta(2r+1,J)$ is isomorphic to $\Delta(2r'+1,J')$ if and only if $J$ equals $J'$ up to cyclic permutation and reversion.
\end{corollary}

\begin{proof}
We already know that if $J$ equals $J'$ up to cyclic permutation and reversion, then $\Delta(2r+1,J)$ is isomorphic to $\Delta(2r'+1,J')$.

The other direction follows from the previous Corollary since the pattern of minimal non-simplices of $\Delta(2r+1,J)$ is, by an isomorphism, mapped to the same pattern of $\min(\Delta(2r'+1,J'))$.
\end{proof}

Note that for $J\in\mathcal J_k$, minimal non-simplices of $\Delta(|J|,J)$ are of length $k+1$ or $k+2$. Hence, if $j_i=2$, then $|\sigma_i|+|\sigma_{i+1}|=2k+4-2=2k+2$, so $|\tau_i|=|\tau_{i+1}|=k+1$, and if $j_i=1$, then one of $|\tau_i|$ and $|\tau_{i+1}|$ is $k+1$ and other is $k+2$.

\begin{lemma}\label{lem:IsubMc}
Let $k\in\NN$, $J\in \mathcal J_k$, and let $c$ be $(k+2,k)$, $(k+1,k+1)$, $(k+1,k,1)$, $(k,k,2)$, or $(k,k,1,1)$. There exists $M=\langle x_1^{a_1}x_2^{b_1},\ldots x_1^{a_r}x_2^{b_r}\rangle$, where $0\le a_1<a_2<\ldots<a_r\le k$ and $k\ge b_1>b_2>\ldots >b_r\ge 0$, such that $M_2^k\subseteq M\subseteq M_2^{k+1}$ and $\Delta(|J|,J)$ is isomorphic to $\Bier_c(M)$.
\end{lemma}

\begin{proof}
We may assume that $j_{2r}=2$. 
If $j_0=1$, then we define $(\alpha_1,\beta_1)=(0,k)$, and if $j_0=2$, then we define $(\alpha_1,\beta_1)=(1,k)$. So $|\sigma_1|=k+3-j_0$, therefore $j_0=2k+4-|\sigma_0|-|\sigma_1|$.
Suppose we already defined $(\alpha_1,\beta_1),\ldots, (\alpha_i,\beta_i)$. We consider the following cases:
\begin{itemize}
\item If $(j_{2i-1},j_{2i})=(1,1)$, then we define $\alpha_{i+1}=\alpha_i+1$ and $\beta_{i+1}=\beta_i-1$. Then by \Cref{lemma_minimalnon} we have $|\sigma_{2i-1}|+|\sigma_{2i}|=2k+3=|\sigma_{2i}|+|\sigma_{2i+1}|$, hence $j_{2i-1}=2k+4-|\sigma_{2i-1}|-|\sigma_{2i}|=1$ and $j_{2i}=2k+4-|\sigma_{2i}|-|\sigma_{2i+1}|=1$.

\item If $(j_{2i-1},j_{2i})=(1,2)$, then by the assumption on $J$ we have $j_{2i-2}=1$.
Suppose that $\alpha_i+\beta_i=k+1$. Then by Lemma~\ref{lemma_minimalnon} we have $\alpha_{i-1}+\beta_{i-q}=k+1$, so $(j_{2i-3},j_{2i-2})=(1,1)$ and again by the assumption on $J$ we have $j_{2i-4}=1$. In $i$ steps we get $\alpha_p+\beta_p=k+1$ for all $p=1,\ldots,2i-1$. But then $j_1=j_{2i}=2$ and $j_q=1$ for all $q=2,\ldots,2i-1$ which is not possible.
So $\alpha_i+\beta_i=k$ and we define $\alpha_{i+1}=\alpha_i+1$ and $\beta_{i+1}=\beta_i-1$ and by \Cref{lemma_minimalnon} we have $j_{2i-1}=2k+4-|\sigma_{2i-1}|-|\sigma_{2i}|=1$ and $j_{2i}=2k+4-|\sigma_{2i}|-|\sigma_{2i+1}|=1$.

\item Let $(j_{2i-1},j_{2i})=(2,1)$. Suppose that $\alpha_i+\beta_i=k$. Then $j_{2i-2}=1$, so but the assumption on $J$, also $j_{2i-3}=1$. By Lemma~\ref{lemma_minimalnon} we have $\alpha_{i-1}+\beta_{i-q}=k$. In the same way, we see that $\alpha_p+\beta_p=k$ for all $p=1,\ldots,i$, so $j_q=0,\ldots,2i-2$. So between $2=j_{2r}$ and the next $2=j_{2i-1}$ there are oddly many 1's, which is not possible. So $\alpha_i+\beta_i=k+1$. We define $\alpha_{i+1}=\alpha_i+2$, $\beta_{i+1}=\beta_i-1$, and we have $j_{2i-1}=2k+4-|\sigma_{2i-1}|-|\sigma_{2i}|=2$ and $j_{2i}=2k+4-|\sigma_{2i}|-|\sigma_{2i+1}|=1$.

\item Let $(j_{2i-1},j_{2i})=(2,2)$. As in the previous case, we see that $\alpha_i+\beta_i=k+1$. We define $\alpha_{i+1}=\alpha_i+2$, $\beta_{i+1}=\beta_i-2$, and we have $j_{2i-1}=2k+4-|\sigma_{2i-1}|-|\sigma_{2i}|=2$ and $j_{2i}=2k+4-|\sigma_{2i}|-|\sigma_{2i+1}|=2$.
\end{itemize}
We constructed $M$ such that $j_{i}=2k+4-|\sigma_{i}|-|\sigma_{i+1}|$ for all $i=2r,0,\ldots,2r-2$ which also implies the last equality $j_{2r-1}=2k+4-|\sigma_{2r-1}|-|\sigma_r|=k+3-|\sigma_{2r-1}|$.
\end{proof}

\begin{lemma}\label{lem:Mcisomorphism}
Let $k\in\NN$, let $c$ be $(k+2,k)$ or $(k+1,k+1)$ and let $M_2^k\subseteq M\subseteq M_2^{k+1}$. There exists $J\in\mathcal J_k$ such that $\Bier_c(M)$ is isomorphic to $\Delta(|J|,J)$.
\end{lemma}

\begin{proof}
Let us first assume that $\alpha_r\le k$ or $\beta_1\le k$ (this always applies in the case of $c=(k+2,k)$). Because of the simmetry in the case $c=(k+1,k+1)$, we may assume that $\beta_1\le k$. The set $\sigma_{2r}=\{\gen x10,\ldots,\gen x1{k+1}\}$ is a minimal non-simplex in $\Bier_c(M)$ since either $\beta_1=k<c_2$, so $x_2^{k+1}$ is a minimal non-element of $M$, either the second component of $c$ equals $c_2=k=\beta_1$ and the polariazation of $x_1^{k+1}$ is a generator of the Stanley--Reisner ideal. In the same way we can see that the set $\sigma_{0}=\{\gen x20,\ldots,\gen x2k\}$ is a minimal non-simplex in $\Bier_c(M)$. Together with Lemma~\ref{lemma_minimalnon} we see that the set of all minimal non-simplices of $\Bier_c(M)$ is $\{\sigma_i\mid i=0,\ldots,2r\}$, where $\sigma_{2i-1}=\{\gen x1{\alpha_i+1},\ldots,\gen x1{c_1},\gen x2{\beta_i+1},\ldots,\gen x2{c_2}\}$ for $i=1,\ldots,r$, and $\sigma_{2i}=\{\gen x10,\ldots,\gen x1{\gamma_i-1},\gen x20,\ldots,\gen x2{\delta_i-1}\}$ for $i=1,\ldots,r-1$. Since $\beta_1=\alpha_r=k$ we have $\sigma_0\cap\sigma_1=\emptyset=\sigma_{2r-1}\cap\sigma_{2r}$. Since $\alpha_i<\gamma_i\le \alpha_{i+1}$ and $\beta_i\ge \delta_i>\beta_{i+1}$ we have $\sigma_i\cap\sigma_{i+1}=\emptyset$ for all $i=0,\ldots,2r$, so $\Bier_c(M)$ is isomorphic to $\Delta(2r+1,J)$, where $j_i=2k+4-|\sigma_{i}|-|\sigma_{i+1}|$. Let us prove that $J=(j_0,\ldots,j_{2r})\in\mathcal J_k$. 

We have $j_{2r}=2k+4-|\sigma_{2r}|-|\sigma_{0}|=2$. If $j_0=2$ between $j_{2r}=2$ and the next $2=j_0$, there are no therefore evenly many 1's. Suppose that $j_0=1$, so $1=j_0=2k+4-|\sigma_0|-|\sigma_1|$. Therefore $\alpha_1+\beta_1=2k+2-|\sigma_1|=|\sigma_0|-1=k$. If $\alpha_i+\beta_i=k$ for all $i$, then by Lemma~\ref{lemma_minimalnon} we have $j_i=1$ for all $i=0,\ldots,2r-1$, so there is evenly many 1's between $2=j_{2r}$ and the next one which in this case is again $2=j_{2r}$. Suppose there exists $i>1$ such that $\alpha_i+\beta_i=k+1$, then $j_p=1$ for all $p=0,\ldots,2i-1$, and $j_{2i}=2$, so there are evenly many 1's between $2=j_{2r}$ and the next $2=j_{2i}$.

Suppose that $j_{2i}=2$ for some $i=1,\ldots,r-1$. By Lemma~\ref{lemma_minimalnon} we have $\alpha_i+\beta_i=\gamma_{i+1}+\delta_{i+1}=k+1$. If $\gamma_p+\gamma_p=k+1$ for all $p=i+1,\ldots, r-1$, then by Lemma~\ref{lemma_minimalnon} we have $j_q=1$ for all $q=2i+1,\ldots,2r-2$, and $j_{2r-1}=2k+4-|\sigma_{2r-1}|-|\sigma_{2r}|=2k+4-(k+1)-(k+1)=2$, so between $2=j_{2i}$ and the next $2=j_{2r-1}$ there are 0 so evenly many 1's. If $\gamma_s+\delta_s=k$ for some $s>i+1$ then by Lemma~\ref{lemma_minimalnon} we have $j_q=1$ for all $q=2i+1,\ldots,2s$ and $j_{2s+1}=2$.

Suppose that $j_{2i-1}=2$ for some $i=1,\ldots,r$. Then by Lemma~\ref{lemma_minimalnon} we have $\alpha_{i-1}+\beta_{i-1}=\gamma_{i-1}+\delta_{i-1}=k+1$. Suppose that $\alpha_p+\beta_p=k$ for all $p=i,\ldots,r$. Then by Lemma~\ref{lemma_minimalnon} we have $j_q=1$ for all $q=2i,\ldots,2r-1$, so between $j_{2i-1}=2$ and the next $2=j_{2r}$ there are evenly many 1's. 
Suppose that there exists $p\ge i$ such that $\alpha_p+\beta_p=k+1$, then by Lemma~\ref{lemma_minimalnon} we have $j_q=1$ for all $q=2i,\ldots,2q-1$ and $j_{2q}=2$. So between $j_{2i-1}=2$ and the next $2=j_{2q}$ there are evenly many 1's. 

We showed that $\Bier_c(M)$  is isomorphic to $\Delta(|J|,J)$ and $J\in\mathcal J_k$.

We need to consider the case when $c=(k+1,k+1)$, $a_r=k+1$, and $b_1=k+1$. Then $x_1^{k+1}$ and $x_2^{k+1}$ are generators of $M$, hence $\sigma_1=\{\gen x11,\ldots,\gen x 1 {k+1}\}$ and $\sigma_{2r-1}=\{\gen x21,\ldots,\gen x 2 {k+1}\}$, so $I_{SR}(\Bier_c(M))=\pol(I_c(M))+\pol^\ast(I_c(M^\vee))$
and the set of minimal simplices of $\Bier_c(M)$ is $\{\sigma_1,\ldots,\sigma_{2r-1}\}$. 
We have $j_{2r-1}=2k+4-|\sigma_{2r-1}|-|\sigma_1|=2$. As in the previous case we can show that $J=(j_1,\ldots,j_{2r-1})\in\mathcal J_k$. 
\end{proof}

\begin{remark}
From the above proof we see that for $c$ equal $(k+2,k)$ or $(k+1,k+1)$ and $M$ be a $c$-multicomplex such that $M_2^k\subseteq M\subseteq M_2^{k+1}$ there exists a $c$-multicompelx $M'=\langle x_1^{a_1}x_2^{b_1},\ldots, x_1^{a_r}x_2^{b_r}\rangle$, where $0\le a_1<a_2<\ldots<a_r\le k$ and $k\ge b_1>b_2>\ldots >b_r\ge 0$, such that $M_2^k\subseteq M\subseteq M_2^{k+1}$ and $\Bier_c(M)$ is isomorphic to $\Bier_c(M')$.
\end{remark}

\begin{corollary}
Let $k\in\NN$, let $c$ be $(k+1,k,1)$, $(k,k,2)$, or $(k,k,1,1)$, and let $M_2^k\subseteq M\subseteq M_2^{k+1}$. For every $c$-multicomplex $M$ such that $M_2^k\subseteq M\subseteq M_2^{k+1}$ we have $\Bier_c(M)\cong\Bier_{(k+2,k)}(M)$ and $\mathcal M^c=\mathcal M^{(k+2,k)}$.
\end{corollary}

\begin{proof}
Let $c=(k,k,1,1)$ and  let $M$ be a $c$-multicomplex such that $M_2^k\subseteq M\subseteq M_2^{k+1}$. Then $M$ can be viewed also as $c'$-multicomplex, where $c'=(k+2,k)$. It is easy to see that $\Bier_c(M)$ and $\Bier_{c'}(M)$ are isomorphic; the isomorphism is $\gen{x}{i}{j}\mapsto \gen{x}{i}{j}$ for $i=1,2$, $\gen{x}{3}{1}\mapsto \gen{x}{1}{k+1}$, and $\gen{x}{4}{1}\mapsto \gen{x}{1}{k+2}$. 
So $\mathcal M^c\subseteq \mathcal M^{(k+2,k)}$. By \Cref{lem:Mcisomorphism} we have $\mathcal M_k^J\subseteq \mathcal M^c$ and by \Cref{lem:IsubMc}, we have $\mathcal M^{(k+2,k)}=\mathcal M_k^J$. In all the remaining cases, we can argue in the same way, which finishes the proof.
\end{proof}

We end up this section with a general result about combinatorial types of simplicial spheres with few vertices.

\begin{theorem}
Let $d\geq 2$. Then any neighborly simplicial $(d-1)$-sphere with no more than $d+3$ vertices is isomorphic to a Murai sphere.     
\end{theorem}
\begin{proof}
Let $d\geq 4$. Due to~\cite{Mani}, any such sphere $K$ is polytopal. Then $f_0(K)$ can take the following 3 values:
\begin{itemize}
\item $f_0(K)=d+1$. Then $K=\partial\Delta^{d}$, and therefore $K=\Bier_c(\langle 1\rangle)$, where $c=(1,1,\ldots,1)\in\N^{m}, m=d+1$.

\item $f_0(K)=d+2$. Then $K=\partial\Delta^{d_1}\ast\partial\Delta^{d_2}$, and therefore, $K=\Bier_c(\langle x^{d_1}\rangle)$, where $c\in\N^m, m=1$, and $c=d_1+d_2+1$ due to Example~\ref{MuraiWithMequalsOneExample}.

\item $f_0(K)=d+3$. Then, by Theorem~\ref{NeigborlyJ}, either $K=\partial\Delta^1\ast\partial\Delta^1\ast\partial\Delta^1$, or $K=\Delta(p,J)$, for some $p\in\N$ and $\{1,2\}$-vector $J\in\N^{2p+3}$ satisfying the Evenness condition. In the former case, $K=\Bier_c(\langle x_1x_2x_3 \rangle)$, where $c\in\N^m, m=4$. In the latter case, we have the following two subcases.

a) If $\dim K = 2k-1$, then $K=\Delta(2k+3,2k)$ and therefore $K=\Bier_{(k+1,k)}(M_2^k)$.

b) If $\dim K = 2k$, then $K$ is isomorphic to a Murai sphere by Theorem~\ref{MuraiClassesCoincideThm}. \end{itemize}

The cases $d=2, 3$ are covered by the classification of low-dimensional Murai spheres given by Theorem~\ref{LowDimClassificationThm}.
\end{proof}

\section{Neighborly Murai spheres: geometric classification}

In \Cref{NeighborlyExplicitClassificationThm} we provided an algebraic classification of all neigborly Murai spheres; that is, we identified the corresponding vectors $c\in\N^m$ and proper $c$-multicomplexes $M$. Here, we are going to complete the geometric classification of all neighborly Murai spheres; that is, we will describe their combinatorial types in terms of a set of simple polytopes, whose face lattice structure is already known, and the simplicial multiwedge operation. 

The case (a) gives all boundaries of simplices. The case (b) gives all joins of boundaries of two simplices of the same dimension, and the case (c) gives all joins of boundaries of two simplices whose dimensions differ by one. These two cases give the same class of neighbourly Murai spheres as the case (g). 

The case (d) gives neighborly spheres of dimension 4 with 9 vertices. The classification of neighborly simplicial 5-polytopes with 9 vertices is done in~\cite{Fin}, and the following example shows that all Murai spheres that come from the case (d) are polytopal.

\begin{example}\label{NeighborlyMuraiTwoTwoTwoExample}
For $c=(2,2,2)$ there are 16 nonisomorphic $c$-multicomplexes $M$ up to duality and symmetry such that $M_2^2\subseteq M\subseteq M_2^3$. In order to classify the combinatorial types of the corresponding Murai spheres, we make use of the classification of simplicial neighborly 5-polytopes with 9 vertices obtained in~\cite{Fin} as well as the notation introduced there.

\begin{enumerate}
\item $M=\langle x^2z,xy^2,yz^2\rangle$: An isomorphism between $\Bier_c(M)$ and $\partial{\mathcal P}_{42}^0$ is\\
$x_{1,0}\mapsto 7,x_{1,1}\mapsto 8,x_{1,2}\mapsto 1,x_{2,0}\mapsto 9,x_{2,1}\mapsto 3,x_{2,2}\mapsto 2,x_{3,0}\mapsto 5,x_{3,1}\mapsto 6,x_{3,2}\mapsto 4$. 
\item $M=\langle x^2z,xy^2,yz,z^2\rangle$: An isomorphism between $\Bier_c(M)$ and $\partial{\mathcal P}_{13}^0$ is\\
$x_{1,0}\mapsto 2,x_{1,1}\mapsto 4,x_{1,2}\mapsto 6,x_{2,0}\mapsto 7,x_{2,1}\mapsto 3,x_{2,2}\mapsto 1,x_{3,0}\mapsto 5,x_{3,1}\mapsto 9,x_{3,2}\mapsto 8$. 
\item $M=\langle x^2,xyz,y^2,z^2\rangle$: An isomorphism between $\Bier_c(M)$ and $\partial{\mathcal P}_{75}^0$ is\\
$x_{1,0}\mapsto 2,x_{1,1}\mapsto 4,x_{1,2}\mapsto 6,x_{2,0}\mapsto 7,x_{2,1}\mapsto 5,x_{2,2}\mapsto 3,x_{3,0}\mapsto 9,x_{3,1}\mapsto 1,x_{3,2}\mapsto 8$.
\item $M=\langle x^2z,xyz,y^2,z^2\rangle$: An isomorphism between $\Bier_c(M)$ and $\partial{\mathcal P}_{28}^0$ is\\
$x_{1,0}\mapsto 9,x_{1,1}\mapsto 1,x_{1,2}\mapsto 8,x_{2,0}\mapsto 2,x_{2,1}\mapsto 4,x_{2,2}\mapsto 6,x_{3,0}\mapsto 7,x_{3,1}\mapsto 5,x_{3,2}\mapsto 3$.
\item $M=\langle x^2z,xy,y^2z,z^2\rangle$: An isomorphism between $\Bier_c(M)$ and $\partial{\mathcal P}_{30}^0$ is\\
$x_{1,0}\mapsto 5,x_{1,1}\mapsto 3,x_{1,2}\mapsto 1,x_{2,0}\mapsto 7,x_{2,1}\mapsto 2,x_{2,2}\mapsto 4,x_{3,0}\mapsto 9,x_{3,1}\mapsto 8,x_{3,2}\mapsto 6$.
\item $M=\langle x^2z,xyz,y^2z,z^2\rangle$: An isomorphism between $\Bier_c(M)$ and $\partial{\mathcal P}_{45}^0$ is\\
$x_{1,0}\mapsto 7,x_{1,1}\mapsto 5,x_{1,2}\mapsto 6,x_{2,0}\mapsto 9,x_{2,1}\mapsto 3,x_{2,2}\mapsto 4,x_{3,0}\mapsto 1,x_{3,1}\mapsto 8,x_{3,2}\mapsto 2$.
\item $M=\langle x^2y,xyz,y^2z,z^2\rangle$: An isomorphism between $\Bier_c(M)$ and $\partial{\mathcal P}_{20}^0$ is\\
$x_{1,0}\mapsto 7,x_{1,1}\mapsto 5,x_{1,2}\mapsto 3,x_{2,0}\mapsto 6,x_{2,1}\mapsto 1,x_{2,2}\mapsto 8,x_{3,0}\mapsto 2,x_{3,1}\mapsto 4,x_{3,2}\mapsto 9$.
\item $M=\langle x^2z,xy^2,y^2z,z^2\rangle$: An isomorphism between $\Bier_c(M)$ and $\partial{\mathcal P}_{32}^0$ is\\
$x_{1,0}\mapsto 7,x_{1,1}\mapsto 2,x_{1,2}\mapsto 4,x_{2,0}\mapsto 5,x_{2,1}\mapsto 1,x_{2,2}\mapsto 8,x_{3,0}\mapsto 3,x_{3,1}\mapsto 9,x_{3,2}\mapsto 6$.
\item $M=\langle x^2y,xy^2,xyz,z^2\rangle$: An isomorphism between $\Bier_c(M)$ and $\partial{\mathcal P}_{18}^0$ is\\
$x_{1,0}\mapsto 2,x_{1,1}\mapsto 4,x_{1,2}\mapsto 6,x_{2,0}\mapsto 3,x_{2,1}\mapsto 1,x_{2,2}\mapsto 8,x_{3,0}\mapsto 7,x_{3,1}\mapsto 5,x_{3,2}\mapsto 9$.
\item $M=\langle x^2z,xy^2,xz^2,yz\rangle$: An isomorphism between $\Bier_c(M)$ and $\partial{\mathcal P}_{37}^0$ is\\
$x_{1,0}\mapsto 3,x_{1,1}\mapsto 1,x_{1,2}\mapsto 6,x_{2,0}\mapsto 5,x_{2,1}\mapsto 9,x_{2,2}\mapsto 8,x_{3,0}\mapsto 7,x_{3,1}\mapsto 2,x_{3,2}\mapsto 4$.
\item $M=\langle x^2y,xz,y^2,yz,z^2\rangle$: An isomorphism between $\Bier_c(M)$ and $\partial{\mathcal P}_{67}^0$ is\\
$x_{1,0}\mapsto 4,x_{1,1}\mapsto 6,x_{1,2}\mapsto 1,x_{2,0}\mapsto 9,x_{2,1}\mapsto 8,x_{2,2}\mapsto 3,x_{3,0}\mapsto 2,x_{3,1}\mapsto 7,x_{3,2}\mapsto 5$.
\item $M=\langle x^2y,x^2z,y^2,yz,z^2\rangle$: An isomorphism between $\Bier_c(M)$ and $\partial{\mathcal P}_{29}^0$ is\\
$x_{1,0}\mapsto 9,x_{1,1}\mapsto 1,x_{1,2}\mapsto 8,x_{2,0}\mapsto 2,x_{2,1}\mapsto 4,x_{2,2}\mapsto 6,x_{3,0}\mapsto 7,x_{3,1}\mapsto 5,x_{3,2}\mapsto 3$.
\item $M=\langle x^2y,xy^2,xz,yz,z^2\rangle$: An isomorphism between $\Bier_c(M)$ and $\partial{\mathcal P}_{79}^0$ is\\
$x_{1,0}\mapsto 4,x_{1,1}\mapsto 6,x_{1,2}\mapsto 2,x_{2,0}\mapsto 7,x_{2,1}\mapsto 5,x_{2,2}\mapsto 1,x_{3,0}\mapsto 9,x_{3,1}\mapsto 8,x_{3,2}\mapsto 3$.
\item $M=\langle x^2y,x^2z,xy^2,yz,z^2\rangle$: An isomorphism between $\Bier_c(M)$ and $\partial{\mathcal P}_{73}^0$ is\\
$x_{1,0}\mapsto 8,x_{1,1}\mapsto 1,x_{1,2}\mapsto 6,x_{2,0}\mapsto 3,x_{2,1}\mapsto 5,x_{2,2}\mapsto 7,x_{3,0}\mapsto 9,x_{3,1}\mapsto 4,x_{3,2}\mapsto 2$.
\item $M=\langle x^2y,x^2z,xyz,y^2,z^2\rangle$: An isomorphism between $\Bier_c(M)$ and $\partial{\mathcal P}_{52}^0$ is\\
$x_{1,0}\mapsto 9,x_{1,1}\mapsto 1,x_{1,2}\mapsto 8,x_{2,0}\mapsto 2,x_{2,1}\mapsto 4,x_{2,2}\mapsto 6,x_{3,0}\mapsto 7,x_{3,1}\mapsto 5,x_{3,2}\mapsto 3$.
\item $M=\langle x^2,xy,xz,y^2,yz,z^2\rangle$: An isomorphism between $\Bier_c(M)$ and $\partial{\mathcal P}_{82}^0$ is\\
$x_{1, 0} \mapsto 8, x_{1, 1}\mapsto 3, x_{1, 2}\mapsto 1, x_{2, 0}\mapsto 4, x_{2, 1} \mapsto 6, x_{2, 2}\mapsto 2, x_{3, 0}\mapsto 9, x_{3, 1}\mapsto 7, x_{3, 2}\mapsto 5$.
\end{enumerate}
Note that no one of these spheres is isomorphic to $\Delta(9,5)$: the number of universal edges in the former ones is less than or equal to 18, while that in the latter one equals 21. 
\end{example}

The case (e) was dealt with in the previous section, where we showed that those neighbouring Murai spheres are precisely spheres that are the boundary of a simplicial multiwedge of a dual cyclic polythope $C^*(2p+3,2p)$ with a vector $J$ satisfying the Evenness condition. The remaining case (f) of \Cref{NeighborlyExplicitClassificationThm} yields the boundary of a cyclic polytope $C(2p+3,2p)$, as shown explicitly in the next example.

\begin{example}
Let $c=(k+1,k)$ and $M=\langle x_1^a x_2^b\mid a+b=k\rangle$. The Murai sphere $\Bier_c(M)$ has dimension $2k-1$ and the set of vertices $V(\Bier_c(M))=\{x_{1,0},\ldots,x_{1,k+1},x_{2,0},\ldots,x_{2,k}\}$. A subset $\sigma\subset V(\Bier_c(M))$ is a facet if and only if $\sigma=V(\Bier_c(M))\setminus\{x_{1,i},x_{2,j_1},x_{2,j_2}\}$ where $i+j_1\le k$ and $i+j_2>k$ or $\sigma=V(\Bier_c(M))\setminus\{x_{1,i_1},x_{1,i_2},x_{2,j}\}$, where $i_1+j\le k$ and $i_2+j>k$. Consider the simplicial sphere $\Delta=\Delta(2k+3,2k)$ and the corresponding set $T=\{t_1,\ldots,t_{2k+3}\}$. The map $f\colon V(\Bier_c(M))\to T$, defined by $f(x_{1,j})=t_{2j+1}$ and $f(x_{2,j})=t_{2(k-j)+2}$ is a bijection. For a facet $\sigma=V(\Bier_c(M))\setminus\{x_{1,i},x_{2,j_1},x_{2,j_2}\}$, $j_1<j_2$, we have $f(\sigma)=T\setminus\{t_{2i+1},t_{2(k-j_1)+2},t_{2(k-j_2)+2}\}$. The element $t_{2(k-j_2)+2}$ is to the left of $t_{2(k-j_1)+2}$. Since $i+j_1\le k$ the element $t_{2i+1}$ is to the left of $t_{2(k-j_1)+2}$ and since $i+j_2 > k$, the element $t_{2i+1}$ is to the right of $t_{2(k-j_2)+2}$. Hence $t_{2i+1}$ is in between the elements $t_{2(k-j_2)+2}$ and $t_{2(k-j_1)+2}$, so there are evenly many elements in $T$ between $t_{2(k-j_2)+2}$ and $t_{2i+1}$. Similarly, there are evenly many between $t_{2i+1}$ and $t_{2(k-j_1)+2}$, so the set $f(\sigma)$ is a facet in $\Delta$. The argument in the case of the facet $\sigma=V(\Bier_c(M))\setminus\{x_{1,i_1},x_{1,i_2},x_{2,j}\}$, where $i_1+j\le k$ and $i_2+j>k$ is similar. Hence $f$ induces an isomorphism between $\Bier_c(M)$ and $\Delta$.

Let $c=(k,k,1)$ and $M=\langle x_1^a x_2^b\mid a+b=k\rangle$. The Murai sphere $\Bier_c(M)$ is also isomorphic to $\Delta(2k+3,2k)$, and the isomorphism is $f\colon V(\Bier_c(M))\to T$, $f(x_{1,i})=t_{2i+1}$, $f(x_{2,i})=t_{2(k-i)+2}$, $f(x_{3,1})=t_{2k+3}$.
\end{example}

Combining all the previous results, we obtain the following classification of all neighborly Murai spheres. Since all spheres in dimensions $1$ and $2$ are neighborly, it remains to consider the case of an $(n-1)$-dimensional Murai sphere with $n\geq 4$.

\begin{theorem}\label{ClassificationNeighMuraiPolytopalThm}
Let $n\geq 4$. An $(n-1)$-dimensional Murai sphere is neighborly if and only if it is isomorphic to one of the following polytopal spheres:
\begin{itemize}
\item $\Delta(n+1,n), n\geq 4$, the boundary of a simplex;
\item $\Delta(n+2,n),n\geq 4$, a join of two boundaries of simplices whose dimensions differ by not more than one;
\item the boundary of a neighborly simplicial 5-polytope with 9 vertices described in Example~\ref{NeighborlyMuraiTwoTwoTwoExample};
\item $\Delta(n+3,n)(J),n\geq 4$, where $n$ is even and $J$ satisfies the \emph{Evenness condition}.
\end{itemize}
In particular, each neighborly Murai sphere is polytopal.
\end{theorem}

Finally, we discuss some applications of the previous result and related open problems. 

\begin{corollary}\label{ClassificationCyclicMuraiCoro}
A Murai sphere is isomorphic to $\Delta(a,b)$ if and only if one of  the following conditions holds:
\begin{itemize}
\item $b=2, 3$ and $1\leq a-b\leq 4$;
\item $b\geq 4$ and $1\leq a-b\leq 3$.  \end{itemize}  
\end{corollary}
\begin{proof}
The implication $\Leftarrow$ is immediate from~\cite[Proposition 3.2, Theorem 3.6]{LV2} and Theorem~\ref{ClassificationNeighMuraiPolytopalThm}. The implication $\Rightarrow$ follows from the facts that $\Delta(a,b)$ is neighborly and $\Delta(m,2\ell+1)=\Delta(m-1,2\ell)(2,1,\ldots,1)$ due to~\cite[Proposition 2.3.33]{Er}, as well as Example~\ref{NeighborlyMuraiTwoTwoTwoExample} and Theorem~\ref{ClassificationNeighMuraiPolytopalThm}.\end{proof}

Due to~\cite[Theorem 3.2]{LTZ2}, each Bier sphere has a \emph{regular realization}, i.e. it is isomorphic to the underlying simplicial complex of a complete regular fan in the ambient Euclidean space. The situation changes when we consider the class of all Murai spheres. Namely, we have the following generalization of~\cite[Theorem 5.3]{LV2}.

\begin{corollary}
In each dimension $d\geq 5$, there exists a $d$-dimensional Murai sphere, which has no regular realization.    
\end{corollary}
\begin{proof}
By the previous theorem, for each $n\geq 6$, there exists an integer vector $c\in\N^m, m\geq 1$ and a proper $c$-multicomplex $M$ such that $\Bier_c(M)\cong\Delta(n+3,n)$. On the other hand, the right-hand side is a sphere of dimension $d$ with $d=n-1\geq 5$, which does not have a regular realization due to~\cite{Hasui}.    
\end{proof}

By~\cite[Theorem 3.6]{LV2}, each Murai sphere in dimensions $1$ and $2$ has a regular realization. Thus, the next question arises naturally.

\begin{problem}
Prove that each $d$-dimensional Murai sphere with $d\leq 4$ has a regular realization, or find a counterexample.     
\end{problem}

Since any Bier sphere has a canonical starshaped realization~\cite{Zivaljevic19,Zivaljevic21} and each neighborly Murai sphere is polytopal by Theorem~\ref{ClassificationNeighMuraiPolytopalThm}, and therefore starshaped, this justifies asking the following question.

\begin{problem}
Prove that any Murai sphere is a i) PL-sphere, or even ii) a starshaped sphere, or find a counterexample. 
\end{problem}

Another related problem arises in toric topology, where to any simplicial complex $K$ one associates a cellular space with a compact torus action $\zk$ called its moment-angle-complex. The latter space acquires an equivariant smooth structure, provided that $K$ is a starshaped sphere; when $K$ is a general triangulated sphere, $\zk$ is merely a topological manifold, see~\cite{TT}. It was shown in~\cite{GLdM} that if $K$ is an odd-dimensional neighborly polytopal sphere different from the boundary of a simplex, then its moment-angle manifold $\zk$ is diffeomorphic to a connected sum of sphere products with two spheres in each product. All manifolds of the latter smooth type among the moment-angle manifolds over Bier spheres were identified in~\cite{LTZ2}. This brings up our final question.

\begin{problem}
Classify all the Murai spheres whose moment-angle manifolds are homotopy equivalent to connected sums of sphere products with two spheres in each product.    
\end{problem}

\subsection*{Acknowledgements}
The authors are grateful to Suyoung Choi, Taras Panov, and Rade \v{Z}ivaljevi\'c for numerous fruitful discussions, valuable comments and suggestions. Limonchenko was supported by the Serbian Ministry of Science, Technological Development and Innovation through the Mathematical Institute of the Serbian Academy of Sciences and Arts.
Vavpeti\v{c} was supported by the Slovenian Research and Innovation Agency program P1-0292.

%\printbibliography
\bibliographystyle{amsplain}
\bibliography{main2}

\end{document}